\documentclass[12pt]{article}

\input xy

\xyoption{all}

\usepackage[utf8]{inputenc}
\usepackage[english]{babel}
\usepackage[T1]{fontenc}

\usepackage{sectsty,tocloft,hyperref,stmaryrd}

\usepackage[svgnames,x11names]{xcolor}

 \hypersetup{
   colorlinks, linkcolor={Blue},
   citecolor={red}, urlcolor={DarkOliveGreen4}
  }

\usepackage{enumitem}
\usepackage{mdwlist}
\usepackage{amssymb,amsfonts,amsbsy,amsthm,amsmath,graphicx,epsfig,bbm}
\usepackage{
times,mathrsfs}
\usepackage{pgf,tikz}
\usetikzlibrary{arrows}
\usepackage{textcomp}
\usepackage[many]{tcolorbox} 

\partfont{\normalsize}

\subsectionfont{\normalsize}

\subsubsectionfont{\mdseries\itshape\normalsize}

\newenvironment{eq}{\begin{equation}}{\end{equation}} 

\theoremstyle{theorem}
\newtheorem{thm}{Theorem}[section]

\newtheorem{lem}[thm]{Lemma}
\newtheorem{prop}[thm]{Proposition}

\newtheorem{quest}[thm]{Question}
\newtheorem{defi}[thm]{Definition}

\theoremstyle{definition}
\newtheorem{rmq}[thm]{Remark}

\newtheorem{exe}[thm]{Example}

\theoremstyle{remark}

\newenvironment{rmq*}
 {\pushQED{\qed}\rmq}
 {\popQED\endrmq}

\newenvironment{exe*}
 {\pushQED{\qed}\exe}
 {\popQED\endexe}

\renewcommand{\bf}[1]{\boldsymbol{#1}}

\newcommand{\bbE}{\mathbb{E}}

\newcommand{\bbN}{\mathbb{N}}

\newcommand{\bbZ}{\mathbb{Z}}

\newcommand{\bfB}{\mathbf{B}}

\newcommand{\bfX}{\mathbf{X}}
\newcommand{\bfY}{\mathbf{Y}}
\newcommand{\bfZ}{\mathbf{Z}}

\newcommand{\bfa}{\mathbf{a}}

\newcommand{\A}{\mathscr{A}}
\newcommand{\calA}{\mathcal{A}}
\newcommand{\B}{\mathscr{B}}
\newcommand{\calB}{\mathcal{B}}
\newcommand{\C}{\mathscr{C}}
\newcommand{\calC}{\mathcal{C}}
\newcommand{\D}{\mathscr{D}}
\newcommand{\E}{\mathscr{E}}
\newcommand{\calE}{\mathcal{E}}
\newcommand{\F}{\mathscr{F}}

\newcommand{\G}{\mathscr{G}}
\newcommand{\calG}{\mathcal{G}}
\newcommand{\scrH}{\mathscr{H}}
\newcommand{\calH}{\mathcal{H}}

\newcommand{\calL}{\mathcal{L}}

\renewcommand{\P}{\mathscr{P}}
\newcommand{\calP}{\mathcal{P}}
\newcommand{\Q}{\mathcal{Q}}

\newcommand{\calT}{\mathcal{T}}

\newcommand{\scrZ}{\mathscr{Z}}

\renewcommand{\a}{\alpha}

\newcommand{\eps}{\varepsilon}

\newcommand{\g}{\gamma}

\renewcommand{\l}{\lambda}

\newcommand{\s}{\sigma}

\renewcommand{\to}{\longrightarrow}
\newcommand{\arr}{\rightarrow}
\renewcommand{\phi}{\varphi}

\newcommand*{\quotient}[2]
{\ensuremath{
    #1/\!\raisebox{-.65ex}{\ensuremath{#2}}}}

\newcommand{\indep}{\perp \!\!\! \perp}

\def\commutatif{\ar@{}[rd]|{\circlearrowleft}}

\definecolor{main}{HTML}{5989cf}
\newtcolorbox{boxB}{
    fontupper = \bf, 
    boxrule = 1.5pt,
    colframe = main,
    rounded corners,
    arc = 5pt   
}

\begin{document}

\title{Introduction to weak Pinsker filtrations}

\author{Séverin Benzoni\\}
\date{}

\maketitle

\begin{abstract}
We introduce the so-called \emph{weak Pinsker dynamical filtrations}, whose existence in any ergodic system follows from the universality of the weak Pinsker property, recently proved by Austin \cite{austin}. These dynamical filtrations appear as a potential tool to describe and classify positive entropy systems. We explore the links between the asymptotic structure of weak Pinsker filtrations and the properties of the underlying dynamical system. A central question is whether, on a given system, the structure of weak Pinsker filtrations is unique up to isomorphism. We give a partial answer, in the case where the underlying system is Bernoulli. We conclude our work by giving two explicit examples of weak Pinsker filtrations.
\end{abstract}

\tableofcontents

\section{Introduction}

In 1958, Kolmogorov and Sinaï introduced the notion of entropy in ergodic theory: the Kolmogorov-Sinaï entropy (or KS-entropy). The  same year, Kolmogorov introduced another important notion: K-systems. He defined a K-system as a dynamical system $\bfX := (X, \A, \mu, T)$ on which there is a generator $\xi$ whose tail $\s$-algebra $\bigcap_{n \geq 1} \sigma(\xi_{]-\infty, -n]})$ is trivial. There is an equivalent definition that is more intrinsic to the system: $\bfX := (X, \A, \mu, T)$ is a K-system if, and only if, every non-trivial observable $\xi_0$ satisfies $h_\mu(\xi, T) > 0$ (a proof of this equivalence, and a more complete presentation of this notion can be found in \cite{ref_K-systemes}). It is also equivalent to assume that the Pinsker factor of the system is trivial, the Pinsker factor being the $\s$-algebra 
$$\Pi_\bfX = \{A \in \A \; | \; h(\mathbbm{1}_A, T) = 0 \}.$$
The Pinsker factor is simply the largest factor of $\bfX$ that is of entropy $0$. Therefore, a K-system has no non-trivial factor of entropy $0$: it is entirely non-deterministic. For example, the most elementary K-systems are the Bernoulli shifts. They are K-systems because i.i.d. processes satisfy Kolmogorov's $0$-$1$ law.

Entropy is an invariant that quantifies the ``chaos'' of a dynamical system, or more precisely its unpredictability, and many of the questions that arose after its discovery were aimed at understanding the structure of this ``chaos''. The first question, which Kolmogorov asked after proving that Bernoulli shifts are K-systems, was whether all K-systems are Bernoulli shifts, which would imply that these chaotic systems have a very simple structure. More general questions then emerged, and we will return to them in the following paragraphs.

The discovery of entropy first led to non-isomorphism results, particularly for Bernoulli shifts: two isomorphic Bernoulli shifts must have the same entropy. The converse of this result, shown by Ornstein \cite{ornstein_iso_bernoulli_fini, ornstein_iso_bernoulli_infini}, is one of the most notable successes of the KS-entropy. But the ramifications of Ornstein's theory go far beyond Bernoulli shifts, and have had a profound impact on the evolution of ergodic theory. We will confine ourselves here to telling the story of the \emph{weak Pinsker property}.

In the early 1960s, Pinsker, then working in Moscow with Kolmogorov, showed that any K-factor of $\bfX := (X, \A, \mu, T)$ is independent of the Pinsker factor $\Pi_\bfX$ (see \cite{pinsker_disjointness}, but this reference is in Russian). Following this result, although the existence of any specific K-factor had not yet been proved, he issued a conjecture (later called the ``Pinsker conjecture''): any system of non-zero entropy is isomorphic to the direct product of its Pinsker factor and a K-system. A few years later, Sinai published \cite{article_sinai} which seemed to confirm this conjecture: he proved the existence of a factor of $\bfX$ isomorphic to a Bernoulli shift of the same entropy as $\bfX$. Given Pinsker's independence result, it would have been sufficient to prove that the factor constructed by Sinaï and the Pinsker factor generate the entire $\s$-algebra $\A$ to obtain a result even stronger than Pinsker's conjecture: $\bfX$ would then be isomorphic to the direct product of its Pinsker factor and a Bernoulli shift. This ``strong Pinsker conjecture'' would also have proved that any K-system is isomorphic to a Bernoulli shift.

But this conjecture turned out to be false: Ornstein published a first example of a non-Bernoulli K-system \cite{exemple_K-sys_non_bernoulli} which contradicts the strong Pinsker conjecture. Following that, many other counterexamples were built, showing that the family of all K-systems is very broad, leaving little hope for a complete classification of those systems. Among all these counterexamples, we can find a construction by Ornstein \cite{K-sys_sans_racine} that can be used to contradict Pinsker's conjecture. Furthermore, he then refined this result by constructing a \textit{mixing} system that does not verify Pinsker's conjecture \cite{contre_exemple_conj_pinsker}. Thus, all the conjectures formulated in the early years of the study of KS-entropy were wrong, revealing a wide variety of possible phenomena.

One of the ramifications of Ornstein's theory can be found in the work of Thouvenot, who, starting in 1975, became interested in relatively Bernoulli systems and developed a ``relative'' version of Ornstein's theory. Following this work, in his 1977 paper \cite{Thouvenot_1977}, he introduced the \emph{weak Pinsker property}: for any $\varepsilon > 0$, $\bfX := (X,  \A, \mu, T)$ is isomorphic to the direct product of a Bernoulli shift $\bfB$ and a system $\bfX_\eps$ of entropy at most $\varepsilon$:
\begin{eq} \label{eq:weak_pinsker_prop}
\bfX \cong \bfX_\eps \otimes \bfB.
\end{eq}

For four decades, however, it was unclear whether all systems verified the weak Pinsker property. But in 2018, Austin published a paper on the subject \cite{austin} in which he proved that all ergodic systems satisfy the weak Pinsker property. 

We can then iterate this splitting operation: take $(\eps_n)_{n \leq -1}$ an increasing sequence of positive numbers such that $\lim_{n\arr -\infty} \eps_n = 0$, and start by splitting $\bfX$ into 
$$\bfX \cong \bfX_{\eps_{-1}} \otimes \bfB_{-1},$$
 then split $\bfX_{\eps_{-1}}$ into 
$$\bfX_{\eps_{-1}} \cong \bfX_{\eps_{-2}} \otimes \bfB_{-2},$$ 
and so on. This yields a sequence of systems $(\bfX_{\eps_n})_{n \leq -1}$ such that, for every $n \leq -1$, $\bfX_{\eps_n}$ is a factor of $\bfX_{\eps_{n+1}}$. By composing the factor maps, it means that each $\bfX_{\eps_n}$ is a factor of $\bfX$, and therefore generates a $T$-invariant $\s$-algebra $\F_n := \s(\bfX_{\eps_n}) \subset \A$. Because of our iterating construction, we see that $\F_n \subset \F_{n+1}$, so the sequence $\F := (\F_n)_{n \leq 0}$ is a filtration. This is what we call a \emph{weak Pinsker filtration on} $\bfX$ (see Definition \ref{def:weak_pinsker_filtr}). 

The purpose of this paper is to discuss the links between weak Pinsker filtrations and the structure of dynamical systems with positive entropy. Weak Pinsker filtrations fall into the category of dynamical filtrations, i.e. filtrations $\F := (\F_n)_{n \leq0}$ on a dynamical system for which each $\s$-algebra $\F_n$ is $T$-invariant. A framework for the study of such filtrations was introduced in \cite{article_PL}, and this will guide our approach of weak Pinsker filtrations. In Section \ref{sect:dyn_filtr}, we introduce the necessary concepts from the theory of dynamical filtrations. This framework is focused on the various possible structures of filtrations whose tail $\s$-algebra $\bigcap_{n\leq 0} \F_n$ is trivial, which is the type of weak Pinsker filtrations that appear on K-systems (see Theorem \ref{thm:structure_filtr_system}). Therefore, the study of weak Pinsker filtrations we suggest would mainly be aimed at classifying K-systems, and in particular non-Bernoulli K-systems.

In Section \ref{sect:weak_pinsker_filtr}, we give an overview of the results and open questions that arise from the study of the properties of weak Pinsker filtrations, and their relation to the structure of the underlying dynamical system. One of those questions concerns the uniqueness, up to isomorphism, of weak Pinsker filtrations. In Section \ref{sect:unique_prob_bernoulli}, we give a partial answer to this uniqueness problem in the case of Bernoulli systems. That section is based on ideas suggested to us by Thouvenot. The main result of this paper is Theorem \ref{thm:unique_weak_pinsker_bernoulli}. Finally, in Section \ref{sect:example_weak_pinsker}, we give explicit examples of weak Pinsker filtrations, in order to give a more concrete meaning to all of those abstract notions.

\section{Weak Pinsker filtrations and related questions}
\label{sect:weak_pinsker_filtr}

In this section, we introduce the notion of weak Pinsker filtrations, the tools necessary to study them and state some of the main questions concerning those filtrations. Since weak Pinsker filtrations are dynamical filtrations, we will use the framework for classifying dynamical filtrations presented in the previous section.

\subsection{Basic notation}

A \emph{dynamical system} is a quadruple $\bfX := (X, \A, \mu, T)$ such that $(X, \A, \mu)$ is a Lebesgue probability space, and $T$ is an invertible measure-preserving transformation. 

Let $\B, \C \subset \A$ be sub-$\sigma$-algebras. We write $\B \subset \C$ mod $\mu$, if for every $B \in \B$, there exists $C \in \C$ such that $\mu(B \Delta C)=0$. Then, $\B = \C$ mod $\mu$ if $\B \subset \C$ mod $\mu$ and $\C \subset \B$ mod $\mu$. We denote $\B \vee \C$ the smallest $\sigma$-algebra that contains $\B$ and $\C$. 
We say that $\B$ is a \emph{factor} $\sigma$-algebra (or $T$-\emph{invariant} $\sigma$-algebra) if $T^{-1}(\B) = \B$ mod $\mu$. Let $\B, \C$ and $\D$ be sub-$\s$-algebras of $\A$. We say that $\B$ and $\C$ are \emph{relatively independent over $\D$} if for any $\B$-measurable bounded function $B$ and $\C$-measurable bounded function $C$
$$\bbE[B C \, | \, \D] = \bbE[B \, | \, \D] \; \bbE[ C \, | \, \D] \; \text{almost surely}.$$
In this case, we write $\B \indep_\D \C$. If $\D$ is trivial, $\B$ and $\C$ are independent, which we denote $\B \indep \C$. 

If we have two systems $\bfX := (X , \A, \mu, T)$ and $\bfY := (Y, \B, \nu, S)$, a \emph{factor map} is a measurable map $\pi: X \to Y$ such that $\pi_*\mu = \nu$ and $\pi \circ T = S \circ \pi$, $\mu$-almost surely. If such a map exists, we say that $\bfY$ is a \emph{factor} of $\bfX$ and we denote $\sigma(\pi) := \pi^{-1}(\B)$ the $\sigma$-algebra generated by $\pi$. Conversely, we also say that $\bfX$ is an extension of $\bfY$ or that $\bfY$ is embedded in $\bfX$. Moreover, if there exist invariant sets $X_0 \subset X$ and $Y_0 \subset Y$ of full measure such that $\pi: X_0 \to Y_0$ is a bijection, then $\pi$ is an \emph{isomorphism} and we write $\bfX \cong \bfY$.

For a given factor $\s$-algebra $\B$, in general, the quadruple $(X, \B, \mu, T)$ is not a dynamical system since $\B$ need not separate points on $X$, and in this case $(X, \B, \mu)$ is not a Lebesgue probability space. However, there always exist a dynamical system $\bfY$ and a factor map $\pi: \bfX \to \bfY$ such that $\s(\pi) = \B$ mod $\mu$. Moreover, this representation is not unique, but for a given factor $\B$, all such systems $\bfY$ are isomorphic and there is a canonical construction to get a system $\quotient{\bfX}{\B}$ and a factor map $p_\B: \bfX \to \quotient{\bfX}{\B}$ such that $\s(p_\B) = \B$ mod $\mu$ (see \cite[Chapter 2, Section 2]{glasner}).

\subsection{Dynamical filtrations}
\label{sect:dyn_filtr}

Let $\bfX := (X, \A, \mu, T)$. A \emph{dynamical filtration} is a pair $(\F, T)$ such that $\F := (\F_n)_{n \leq 0}$ is a filtration in discrete negative time (i.e. $\F_n \subset \F_{n+1}$) on $\A$ and each $\F_n$ is $T$-invariant. The theory of dynamical filtrations was initiated by Paul Lanthier in \cite{these_PL, article_PL}. For our present work, the primary notion we need is a precise definition of what it means for two dynamical filtrations to be isomorphic:

\begin{defi} \label{def:iso_filtration_dyn}
Let $(\F, T_1)$ be a dynamical filtration on $\bfX_1 := (X_1, \A_1, \mu_1, T_1)$ and $(\G, T_2)$ a dynamical filtration on $\bfX_2 := (X_2, \A_2, \mu_2, T_2)$. We say that $(\F, T_1)$ and $(\G, T_2)$ are isomorphic if there is an isomorphism $\Phi: \quotient{\bfX_1}{\F_0} \arr \quotient{\bfX_2}{\G_0}$ such that, for all $n \leq 0$, $\Phi(\F_n) = \G_n$ mod $\mu_2$.
\end{defi}
We will discuss two specific families of filtrations:
\begin{defi} \label{defi:standard_filtration}
Let $(\F, T)$ be a dynamical filtration on $\bfX := (X, \A, \mu, T)$. It is of product type if there is a sequence $(\C_n)_{n \leq 0}$ of mutually independent factor $\s$-algebras such that 
$$\forall n \leq 0, \; \F_n = \bigvee_{k \leq n} \C_k \; \text{ mod } \, \mu.$$
\end{defi}
\begin{defi}
Let $(\F, T)$ be a dynamical filtration on $\bfX := (X, \A, \mu, T)$. It is Kolmogorovian if its tail $\s$-algebra is trivial, i.e. if $\bigcap_{n \leq 0} \F_n = \{\varnothing, X\}$ mod $\mu$.
\end{defi}
In particular, because of Kolmogorov's $0-1$ law, product type filtrations are Kolmogorovian. 

In the theory of dynamical filtrations, additional properties, such as \emph{standardness} and \emph{I-cosiness}, appear naturally (see \cite{article_PL}). However, for now, we are not able to find relevant applications of those notions in the study of weak Pinsker filtrations, so we dot not discuss them in this paper. That being said, standardness and I-cosiness being looser than ``product-type'', a first step in investigating further the examples of Section \ref{sect:example_weak_pinsker} could be to determine whether they are standard, I-cosy or neither. 

\subsection{Reminders on KS-entropy}

The notion of entropy first appeared in mathematics in 1948, introduced by Shannon in his foundational work on information theory \cite{shannon}. It is defined as follows:
\begin{defi}[Shannon entropy]
Let $(X, \A, \mu)$ be a probability space and $\xi: X \arr A$ a random variable, with $A$ finite or countable. The Shannon entropy of $\xi$ is
$$H_\mu(\xi) := - \sum_{a \in A} \mu(\{\xi = a\}) \cdot \log\mu(\{\xi = a\}).$$
\end{defi}
\noindent The number $H_\mu(\xi)$ gives the average amount of information given by the random variable $\xi$. If we have a probability measure $\rho$ defined directly on $A$, we can also define the entropy of that measure 
$$H(\rho) := - \sum_{a \in A} \rho(a) \cdot \log\rho(a).$$
One can also define conditional entropy: for $\chi_1: X \arr Y_1$ and $\chi_2: X \arr Y_2$ be two random variables, we define 
$$H_\mu(\chi_1 \, | \, \chi_2) := \sum_{y_2 \in Y_2} \mu(\chi_2 = y_2) \sum_{y_1 \in Y_1} \phi(\mu(\chi_1 = y_1 \, | \, \chi_2 = y_2)),$$ 
where $\phi(x) = - x \cdot \log(x)$. This quantifies the missing information required to determine $\chi_1$ once $\chi_2$ is known. We refer to \cite[Chapter 2, Section 6]{book_billingsley} for the basic properties of this notion. In the present work, conditional entropy will only serve as a computational tool, via Fano's inequality. See \cite[Theorem 6.2]{book_billingsley} for a proof.
\begin{lem}[Fano's inequality] \label{lem:Fano}
Let $\chi_1: X \arr A$ and $\chi_2: X \arr A$ be two $A$-valued random variables. Set $d := \mu(\chi_1 \neq \chi_2)$. We have
$$H_\mu(\chi_1 \, | \, \chi_2) \leq \phi(d) + \phi(1-d) + d \log(\#A-1).$$
In particular, for $\eps \in ]0, e^{-1}[$, if $\chi$ is an $A$-valued random variable such that there exists $a_0 \in A$ satisfying $\mu(\chi = a_0) \geq 1-\varepsilon$, then
$$H_\mu(\chi) \leq \varepsilon(1 + \log(\varepsilon^{-1}) + \log(\#A-1)).$$
\end{lem}

In 1958, Kolmogorov and Sinaï used entropy to quantify the randomness of measure preserving dynamical systems. 

Let $\bfX := (X, \A, \mu, T)$ be a dynamical system. To any random variable $\xi_0: X \arr A$, with $A$ finite, we associate $\xi: X \arr A^\bbZ$ the corresponding $T$-process
$$\xi := (\xi_n)_{n \in \bbZ} := (\xi_0 \circ T^n)_{n \in \bbZ}.$$
Also, for $F \subset \bbZ$, set $\xi_F := (\xi_n)_{n \in F}$.

The Kolmogorov-Sinaï entropy (or KS-entropy) of a dynamical system is:
\begin{defi}[Kolmogorov-Sinaï entropy]
Let $\bfX := (X, \A, \mu, T)$ be a dynamical system. For a finite valued random variable $\xi_0: X \arr A$, define
$$h_\mu(\xi, T) := \lim_{n \arr \infty} \frac{1}{n} H_\mu(\xi_{\llbracket 0, n \rrbracket}).$$
For a $T$-invariant $\s$-algebra $\B \subset \A$, define
$$h_\mu(\B, T) := \sup \{h_\mu(\xi, T)\, ; \; \xi_0 \text{ a } \B\text{-measurable random variable} \}.$$
Finally, set
$$h(\bfX) := h_\mu(\A, T).$$
\end{defi}
The KS-entropy satisfies the following continuity result:
\begin{lem} \label{lem:continuity_entropy}
Let $\bfX := (X, \A, \mu, T)$ be a dynamical system and a random variable $\xi_0: X \arr A$, with $A$ finite. For $\eps > 0$, there exists $\delta > 0$ such that, for any random variable $\zeta_0: X \arr A$ such that $\mu(\zeta_0 \neq \xi_0) \leq \delta$, we have 
$$|h_\mu(\xi, T) - h_\mu(\zeta, T)| \leq \eps.$$
\end{lem}
\begin{proof}
In this proof, we will use Fano's inequality (Lemma \ref{lem:Fano}). Specifically, we compute:
\begin{align*}
h(\xi, T) &= \lim_{n \arr \infty} \frac{1}{n} H_\mu(\xi_{[0, n[}) \leq \lim_{n \arr \infty} \frac{1}{n} H_\mu((\xi \vee \zeta)_{[0, n[})\\
&\leq \lim_{n \arr \infty} \frac{1}{n} \left( H_\mu(\zeta_{[0, n[}) + \sum_{j=0}^{n-1} H_\mu(\xi_j \, | \, \zeta_{[0, n[}) \right)\\
&\leq \lim_{n \arr \infty} \frac{1}{n} \left( H_\mu(\zeta_{[0, n[}) + \sum_{j=0}^{n-1} H_\mu(\xi_j \, | \, \zeta_{j}) \right)\\
&\leq h_\mu(\zeta, T) + H_\mu(\xi_0 \, | \, \zeta_0)\\
&\leq h_\mu(\zeta, T) + \phi(d) + \phi(1-d) + d \log(\#A-1).
\end{align*}
where $\phi(x) = - x \cdot \log(x)$. And, since $\phi$ is continuous, there exists $\delta>0$ such that, if $d \leq \delta$, we have
$$h_\mu(\xi, T) \leq h_\mu(\zeta, T) + \eps.$$
By switching $\xi$ and $\zeta$ and doing the same reasoning, we end the proof.
\end{proof}

It is useful to locate the deterministic aspects of a dynamical system. We do that by considering the Pinsker factor of a system. For any factor $\s$-algebra $\B$, we define
$$\Pi_\B = \{A \in \B \; | \; h(\mathbbm{1}_A, T) = 0 \}.$$
The Pinsker factor of the system $\bfX := (X, \A, \mu, T)$ is then defined as $\Pi_\bfX := \Pi_\A$. We will use the following basic result, which can be found in \cite[Theorem 14]{Topics_Parry}:
\begin{lem} \label{lem:pinsker_product}
Let $\bfX := (X, \A, \mu, T)$ be a dynamical system and $\B$ and $\C$ be independent factor $\s$-algebras. We have
$$\Pi_{\B \vee \C} = \Pi_{\B} \vee \Pi_{\C}.$$
\end{lem}

To be able to compute the entropy of a system, the following result proves to be most useful.
\begin{thm}[Kolmogorov-Sinaï] \label{thm:kolmogorov-sinai}
Let $\bfX := (X, \A, \mu, T)$ be a dynamical system. Consider a finite valued random variable $\xi_0: X \arr A$ and the corresponding $T$-process $\xi := (\xi_0 \circ T^n)_{n \in \bbZ}$. Then we have 
$$h_\mu(\s(\xi), T) = h_\mu(\xi, T).$$
In particular, if $\xi$ is a generator of $\A$ (i.e. $\A = \s(\xi)$ mod $\mu$), then $h(\bfX) = h_\mu(\xi, T)$.
\end{thm}

\subsection{Ornstein's theory and its relative version}
\label{sect:entropy_ornstein}

From the definition, one easily sees that KS-entropy is invariant under isomorphism of dynamical systems, which makes it a useful tool in the classification of measure preserving dynamical systems. The most remarkable classification results concern Bernoulli and relatively Bernoulli systems:
\begin{defi}[Bernoulli and relatively Bernoulli]
Let $\bfX := (X, \A, \mu, T)$ be a dynamical system. 

We say that $\bfX$ (or $\A$) is \emph{Bernoulli} if there exists a random variable $\xi_0: X \arr A$ such that the corresponding $T$-process $\xi := (\xi_0 \circ T^n)_{n \in \bbZ}$ is i.i.d. and generates $\A$, i.e. we have $\s(\xi) = \A$ mod $\mu$. 

Let $\B \subset \A$ be a factor $\s$-algebra. We say that $\bfX$ (or $\A$) is \emph{relatively Bernoulli over $\B$} if there is an i.i.d. process of the form $\xi := (\xi_0 \circ T^n)_{n \in \bbZ}$ such that $\s(\xi)$ is independent of $\B$ and $\A = \B \vee \s(\xi)$ mod $\mu$. 
\end{defi}
Those two definitions coincide when $\B$ is the trivial factor $\s$-algebra: $\bfX$ is relatively Bernoulli over $\{\varnothing, X\}$ if and only if $\bfX$ is Bernoulli.
\begin{rmq}
We can consider another approach to define Bernoulli systems: take $A$ a finite or countable set and $\rho$ a probability measure on $A$. On the product probability space $(A^\bbZ, \rho^{\otimes \bbZ})$, consider the transformation
$$S: (a_n)_{n \in \bbZ} \mapsto (a_{n+1})_{n \in \bbZ}.$$
The map $S$ is called the \emph{shift} on $A^\bbZ$. One can easily check that $\rho^{\otimes \bbZ}$ is $S$-invariant. Therefore, this yields a measure preserving dynamical system 
\begin{eq} \label{eq:bernoulli_shift}
\bfB := (A^\bbZ, \rho^{\otimes \bbZ}, S),
\end{eq}
which is called a \emph{Bernoulli shift}. Then a system is Bernoulli if and only if it is isomorphic to a Bernoulli shift. Similarly, we can see that a system $\bfX$ is relatively Bernoulli over a factor $\s$-algebra $\B$ if and only if $\bfX$ is isomorphic to a system of the form $\bfY \otimes \bfB$ via a factor map $\phi: \bfX \to \bfY \times \bfB$ such that $\s(\pi_\bfY \circ \phi) = \B$ mod $\mu$ (where $\pi_\bfY$ is the projection of $\bfY \otimes \bfB$ onto $\bfY$). 
\end{rmq}

Using Theorem \ref{thm:kolmogorov-sinai}, it is easy to compute the entropy of a Bernoulli system. Let $\xi$ be an $i.i.d.$ process on $\bfX$ that generates $\A$. We then have
\begin{align*}
h(\bfX) &= h_\mu(\A, T) = h_\mu(\xi, T) = \lim_{n \arr \infty} \frac{1}{n} H_\mu(\xi_{\llbracket 0, n \rrbracket})\\
&= \lim_{n \arr \infty} \frac{1}{n} \sum_{i=0}^n H_\mu(\xi_i) = H_\mu(\xi_0).
\end{align*}
In particular, if $\bfX$ is isomorphic to a system of the form \eqref{eq:bernoulli_shift}, we get 
$$h(\bfX) = h(\bfB) = -\sum_{a \in A} \rho(a) \log(\rho(a)).$$
Since, to be isomorphic, two systems need to have the same entropy, this computation enables us to get a non-isomorphism result between any two Bernoulli systems of different entropy. Remarkably, Ornstein proved that the converse is also true:
\begin{thm}[Ornstein \cite{ornstein_iso_bernoulli_fini, ornstein_iso_bernoulli_infini}]
If $\bfX$ and $\bfY$ are Bernoulli systems such that $h(\bfX) = h(\bfY)$, then $\bfX \cong \bfY$. 
\end{thm}
This means that the KS-entropy gives a complete classification of Bernoulli systems. An outstanding result that emerged from Ornstein's theory was a criterion to characterize Bernoulli systems: finite determination. However, although this notion is useful in proving abstract results, when studying a given system, it is not easy to know whether or not it is finitely determined. Because of that, another criterion called \emph{very weak Bernoullicity} was developed (see \cite[Section 7]{book_ornstein}). This is the criterion we are interested in. 

For the remainder of this section, we assume that the processes are defined on finite alphabets. We first need a technical definition. Given a finite alphabet $A$, an integer $\ell \geq 1$ and two words $\mathbf{a}, \mathbf{b} \in A^\ell$ of length $\ell$ on $A$, we define the normalized Hamming metric between $\mathbf{a}$ and $\mathbf{b}$ as:
$$d_\ell(\mathbf{a}, \mathbf{b}) := \frac{1}{\ell} \#\{i \in \llbracket 1, \ell \rrbracket \, | \, a_i \neq b_i\},$$
where $\mathbf{a} = (a_1, ..., a_\ell)$ and $\mathbf{b} = (b_1, ..., b_\ell)$. We then consider the corresponding transportation metric on $\P(A^\ell)$:
$$\forall \mu, \nu \in \P(A^\ell), \; \bar{d_\ell}(\mu, \nu) := \inf\left\{ \int d_\ell(\mathbf{a}, \mathbf{b}) d\l(\mathbf{a}, \mathbf{b}) \, ; \; \l \text{ a coupling of $\mu$ and $\nu$}\right\}.$$
Then a process $\xi$ is said to be very weak Bernoulli if, for some $\ell \geq 1$, the conditional law of $\xi_{[0, \ell[}$ given the past of $\xi$ is close enough to the law of $\xi_{[0, \ell[}$ in the $\bar{d_\ell}$ metric. More formally, we state:

\begin{defi}[Very weak Bernoulli]  \label{def:VWB}
Let $\bfX := (X, \A, \mu, T)$ be an ergodic dynamical system, equipped with a $T$-process $\xi$ taking values in a finite alphabet. We say that $\xi$ is very weak Bernoulli if, for every $\varepsilon > 0$, there exists $\ell \geq 1$ such that for every $m \geq 1$, we have
$$\int \bar{d_\ell}\left(\nu_\ell(\cdot \, | \, \mathbf{a}_{[-m, 0[}), \nu_\ell(\cdot)\right) d\nu(\mathbf{a}) \leq \varepsilon,$$
where $\nu$ is the law of $\xi$ and, for $I \subset \bbZ$, $\nu_\ell(\cdot\, | \, \mathbf{a}_I)$ is the conditional law of $\xi_{[0, \ell[}$ given that $\xi_I$ equals $\mathbf{a}_I$.

If $\A = \s(\xi)$, we say that $\bfX$ (or $\A$) is very weak Bernoulli.
\end{defi}
The fact that very weak Bernoullicity characterizes Bernoulli systems can be stated as follows:
\begin{thm}[see \cite{book_ornstein, FD_implies_VWB}] \label{thm:FD_eq_VWB}
Let $\bfX := (X, \A, \mu, T)$ be a dynamical system. A $T$-process $\xi$ on $\bfX$ is very weak Bernoulli if and only if $\s(\xi)$ is Bernoulli.
\end{thm}

Following the work of Ornstein, Thouvenot studied relatively Bernoulli systems and adapted the definitions of finite determination and very weak Bernoullicity to get criteria that characterize relatively Bernoulli systems. Here we give his adaptation of very weak Bernoullicity:
\begin{defi}[Relatively very weak Bernoulli]  \label{def:VWB_rel}
Let $\bfX := (X, \A, \mu, T)$ be an ergodic dynamical system, equipped with two $T$-processes $\xi$ and $\eta$ with finite alphabets. We say that $\xi$ is \emph{relatively very weak Bernoulli over} $\eta$ if, for every $\varepsilon > 0$, there exists $\ell \geq 1$ such that for every $m \geq 1$ and for all $k \geq 1$ large enough, we have
$$\int \bar{d_\ell}\left(\nu_\ell(\cdot \, | \, \mathbf{a}_{[-m, 0[}, \mathbf{b}_{[-k, k]}), \nu_\ell(\cdot \, | \, \mathbf{b}_{[-k, k]})\right) d\nu(\mathbf{a}, \mathbf{b}) \leq \varepsilon,$$
where $\nu$ is the law of $(\xi, \eta)$ and, for $I, J \subset \bbZ$, $\nu_\ell(\cdot\, | \, \mathbf{a}_I, \mathbf{b}_J)$ is the conditional law of $\xi_{[0, \ell[}$ given that $\xi_I$ equals $\mathbf{a}_I$ and that $\eta_J$ equals $\mathbf{b}_J$.

If $\A = \s(\xi)$ and $\B = \s(\eta)$, we say that $\bfX$ (or $\A$) is relatively very weak Bernoulli over $\B$.
\end{defi}
Many early results from Thouvenot's theory were stated for relatively finitely determined systems. However, for our work, relative very weak Bernoullicity is a more convenient notion. Fortunately, we have the following equivalence, which enables us to apply to relatively very weak Bernoulli processes results originally stated for relatively finitely determined processes:
\begin{thm}[see \cite{RFD_implies_RVWB}] \label{thm:rfd_eq_rvwb}
Let $\bfX := (X, \A, \mu, T)$ be an ergodic system and $\xi$ and $\eta$ be $T$-processes with finite alphabets defined on $\bfX$. Then $\xi$ is relatively very weak Bernoulli over $\eta$ if and only if it is relatively finitely determined over $\eta$. 
\end{thm}
We give a summary of the results we will use:
\begin{lem} \label{lem:factors_rel_B_are_B}
Let $\bfX := (X, \A, \mu, T)$ be a finite entropy dynamical system and $\B$ a factor $\s$-algebra. Let $\xi$ and $\eta$ be $T$-processes with finite alphabets defined on $\bfX$ such that $\A = \s(\xi)$  and $\B = \s(\eta)$. If $\xi$ is relatively very weak Bernoulli over $\eta$, then
\begin{enumerate}[label = (\roman*)]
\item $\bfX$ is relatively Bernoulli over $\B$,
\item any $T$-process $\rho$ on $\bfX$ is relatively very weak Bernoulli over $\eta$,
\item for any factor $\s$-algebra $\C \subset \A$, $\B \vee \C$ is relatively very weak Bernoulli over $\B$,
\item any factor $\s$-algebra $\C \subset \A$ that is \emph{independent from $\B$} is Bernoulli.
\end{enumerate}
\end{lem}
\begin{proof}
We prove the lemma mainly by referring to the literature. The statement (i) follows from \cite[Proposition 5]{article_thouvenot_RFD} and Theorem \ref{thm:rfd_eq_rvwb}. Then (ii) follows from \cite[Proposition 4]{article_thouvenot_RFD} and Theorem \ref{thm:rfd_eq_rvwb}, and (iii) follows from (ii). Let us prove (iv): take $\rho$ a process on $\bfX$ such that $\C = \s(\rho)$ mod $\mu$. From (ii), we know that $\rho$ is relatively very weak Bernoulli over $\eta$. However, since $\C$ is independent of $\B$, $\rho$ is independent of $\eta$. One can then notice that if we add this independence in the definition of relative very weak Bernoullicity, we end up with the fact that $\rho$ is very weak Bernoulli. Finally, Theorem \ref{thm:FD_eq_VWB} tells us that $\C = \s(\rho)$ is Bernoulli. 
\end{proof}

We have just given many definitions and results concerning processes with finite alphabets, and the $\s$-algebras they generate. The following result from Krieger tells that it is applicable on any finite entropy system:
\begin{thm}[See \cite{ref_thm_krieger}] \label{thm:krieger}
Let $\bfX := (X, \A, \mu, T)$ be an ergodic dynamical system and $\B \subset \A$ be a factor $\s$-algebra. If $h_\mu(\B, T) < \infty$, there exists a finite alphabet $A$ and a random variable $\xi_0: X \arr A$ such that 
$$\B = \s(\{\xi_0\circ T^n\}_{n \in \bbZ}) \text{ mod } \mu.$$
We say that $\xi$ is a finite generator of $\B$.
\end{thm}

\subsection{Positive entropy systems and weak Pinsker filtrations}

In 2018, Austin proved the following:
\begin{thm}[Austin, 2018, \cite{austin}]
Let $\bfX := (X, \A, \mu, T)$ be an ergodic dynamical system. For every $\eps > 0$ there exists a factor $\s$-algebra $\B$ such that: 
\begin{itemize}
\item $h_\mu(\B, T) \leq \eps$,
\item $\bfX$ is relatively Bernoulli over $\B$.
\end{itemize}
In other words, $\bfX$ has the \emph{weak Pinsker property} (as in \eqref{eq:weak_pinsker_prop}).
\end{thm}

\begin{defi} \label{def:weak_pinsker_filtr}
Let $\bfX := (X, \A, \mu, T)$ be a dynamical system and $\F := (\F_n)_{n \leq 0}$ a dynamical filtration on $\bfX$ such that $\F_0 = \A$. We say that $\F$ is a \emph{weak Pinsker filtration} if 
\begin{itemize}
\item for every $n \leq -1$, $\F_{n+1}$ is relatively Bernoulli over $\F_n$,
\item and $$\lim_{n \arr -\infty} h_\mu(\F_n, T) = 0.$$
\end{itemize}
\end{defi}

Then, by iterating Austin's theorem, we see that we can obtain weak Pinsker filtrations on any ergodic system:
\begin{prop} \label{prop:pinsker_filtrations_exist}
Let $\bfX := (X, \A, \mu, T)$ be a dynamical system. If $\bfX$ is ergodic, there exists a weak Pinsker filtration on $\bfX$. More specifically, for every increasing sequence $(\eps_n)_{n \leq -1}$ such that $\eps_{-1} \leq h(\bfX)$ and $\lim_{n \arr -\infty}\eps_n = 0$, there exists a weak Pinsker filtration $(\F_n)_{n \leq 0}$ such that $\forall n \leq -1$, $h_\mu(\F_n, T) = \eps_n$.
\end{prop}

This simply tells us that weak Pinsker filtrations exist, but gives no explicit description. To start understanding those filtrations better, we can first link them to the Pinsker factor of the system:
\begin{prop} \label{prop:tail_pinsker_filtr_pinsker_factor}
Let $\bfX := (X, \A, \mu, T)$ be a dynamical system and $\F := (\F_n)_{n \leq 0}$ a weak Pinsker filtration on $\bfX$. Then the tail $\s$-algebra $\F_{-\infty} := \bigcap_{n \leq 0} \F_n$ is the Pinsker factor of $\bfX$.
\end{prop}
\begin{proof}
Let $\F := (\F_n)_{n \leq 0}$ be a weak Pinsker filtration on $\bfX$. Since, for $n_0 \leq 0$, $\F_{-\infty} \subset \F_{n_0}$, it follows that $h_\mu(\F_{-\infty}, T) \leq h_\mu(\F_{n_0}, T)$. Then, by taking $n_0 \arr -\infty$, this yields $h_\mu(\F_{-\infty}, T) = 0$. Therefore, $\F_{-\infty} \subset \Pi_\bfX$. 

Conversely, let us show that, for every $n \leq 0$, $\Pi_\bfX \subset \F_n$. Since $\F$ is a weak Pinsker filtration, we can choose $\B_n \subset \A$ a Bernoulli factor $\s$-algebra such that 
$$\F_n \indep \B_n \; \text{ and } \; \F_n \vee \B_n = \A \text{ mod } \mu.$$
Then we use Lemma \ref{lem:pinsker_product}: 
$$\Pi_\bfX = \Pi_\A = \Pi_{\F_n} \vee \Pi_{\B_n} = \Pi_{\F_n} \subset \F_n,$$
because, $\B_n$ being Bernoulli, its Pinsker factor is trivial. 
\end{proof}
Weak Pinsker filtrations are dynamical filtrations, and in Section \ref{sect:dyn_filtr}, we introduced tools to classify dynamical filtrations, which we use here. While trying to connect the properties of a weak Pinsker filtration with the properties of the underlying system, we get the following simple results:
\begin{thm} \label{thm:structure_filtr_system}
Let $\bfX := (X, \A, \mu, T)$ be a dynamical system and $\F := (\F_n)_{n \leq 0}$ be a weak Pinsker filtration on $\bfX$. Then
\begin{enumerate} [label = (\roman*)]
\item $\bfX$ is a K-system if and only if $\F$ is Kolmogorovian, i.e. $\bigcap_{n \leq 0} \F_n = \{\varnothing, X\}$ mod $\mu$.
\item If the filtration $\F$ is of product-type, then $\bfX$ is Bernoulli.
\end{enumerate}
\end{thm}
\begin{proof}
We know that a system is K if and only if its Pinsker factor is trivial. Then the equivalence in (i) follows from Proposition \ref{prop:tail_pinsker_filtr_pinsker_factor}. 

We now prove (ii). Assume that $\F$ is a weak Pinsker filtration of product type. This means that there exists a sequence $(\B_n)_{n \leq 0}$ of mutually independent factor $\s$-algebras such that $\F_n = \bigvee_{k \leq n} \B_k$. Let $n \leq 0$. We know that $\F_n$ is relatively Bernoulli over $\F_{n-1}$ and that $\B_n$ is independent of $\F_{n-1}$. So, Lemma \ref{lem:factors_rel_B_are_B} tells us that $\B_n$ is Bernoulli. Therefore, we have $\A = \F_0 = \bigvee_{k \leq 0} \B_k$, which shows that we can write $\A$ as a product of mutually independent Bernoulli factors. Hence, $\A$ is Bernoulli.
\end{proof}
However, this result leaves many open questions. First, we can ask if the converse of (ii) is true. Since we remark at the end of Section \ref{sect:uniqueness} that, on a Bernoulli shift, there is at least one weak Pinsker filtration of product type, the converse of (ii) is equivalent to the uniqueness problem given in Question \ref{quest:uniqu_problem_bernoulli}. Another area that is left open is to consider other properties from the theory of dynamical filtrations, like standardness or I-cosiness, and wonder what it implies of the system if a weak Pinsker filtrations has those properties:
\begin{quest}
What can we say about $\bfX := (X, \A, \mu, T)$ if there is a weak Pinsker filtration $\F$ on $\bfX$ that is standard ? In that case, is $\bfX$ Bernoulli ? And if the weak Pinsker filtration is I-cosy ?
\end{quest}
\noindent Our hope is that answering those questions could give additional information on the structure of non-Bernoulli K-systems. For precise definitions of standardness and I-cosiness, see \cite{article_PL} or \cite{article_confined}.

\subsection{The uniqueness problem}
\label{sect:uniqueness}
Let $\bfX := (X, \A, \mu, T)$ be an ergodic dynamical system. As mentioned in Proposition \ref{prop:pinsker_filtrations_exist}, the fact that every ergodic systems satisfies the weak Pinsker property implies that, for any given increasing sequence $(\eps_n)_{n \leq -1}$ that goes to $0$ in $-\infty$ such that $\eps_{-1} \leq h(\bfX)$, there exits a weak Pinsker filtration $\F$ on $\bfX$ such that $h_\mu(\F_n, T) = \eps_n$. But this filtration is not unique. Indeed, in the splitting result given by the weak Pinsker property \eqref{eq:weak_pinsker_prop}, the choice of the factor $\s$-algebra generated by $\bfX_{\eps}$ is not unique. For example, take a system of the form 
$$\bfX := \bfZ \otimes \bfB_1 \otimes \bfB_2,$$
where $\bfZ$ is a $0$ entropy system and $\bfB_1$ and $\bfB_2$ are Bernoulli shifts of equal entropy. Note that $\bfZ \otimes \bfB_1$ and $\bfZ \otimes \bfB_2$ generate two different factor $\s$-algebras on $\bfX$. But they are both factors over which $\bfX$ is relatively Bernoulli, and they have the same entropy. However, we can notice in this example that $\bfZ \otimes \bfB_1$ and $\bfZ \otimes \bfB_2$ are isomorphic. This observation hints to a general result:

\begin{thm}[From Thouvenot in \cite{Thouvenot_2008}]
Let $\bfX := (X, \A, \mu, T)$ and $\bfY := (Y, \B, \nu, S)$ be ergodic dynamical systems and $\bfB$ be a Bernoulli shift of finite entropy. If $\bfX \otimes \bfB$ and $\bfY \otimes \bfB$ are isomorphic, then $\bfX$ and $\bfY$ are isomorphic. 
\end{thm}
\begin{proof}
This proof relies on the weak Pinsker property of $\bfX$ and $\bfY$, and Lemma \ref{lem:factors_rel_B_are_B}. We also use many times that Bernoulli shifts with the same entropy are isomorphic.

Since $\bfX \otimes \bfB$ and $\bfY \otimes \bfB$ are isomorphic, we have:
\begin{align*}
h(\bfX) =  h(\bfX \otimes \bfB) - h(\bfB) = h(\bfY \otimes \bfB) - h(\bfB) = h(\bfY).
\end{align*}
Set $a := h(\bfX) = h(\bfY)$. We can then apply the weak Pinsker property of $\bfX$ and $\bfY$ to find two systems $\hat\bfX$, $\hat\bfY$ and a Bernoulli shift $\hat\bfB$ such that 
$$h(\hat\bfX) = h(\hat\bfY) \leq a/3,$$
and
$$\bfX \cong \hat\bfX \otimes \hat\bfB \; \text{ and } \; \bfY \cong \hat \bfY \otimes \hat \bfB.$$
This implies 
$$\hat\bfX \otimes (\hat\bfB \otimes \bfB) \cong \hat \bfY \otimes ( \hat \bfB \otimes \bfB).$$ 
In other words, there is a system $\bfZ$ and two factor maps $p_{\hat \bfX}: \bfZ \to \hat\bfX$ and $p_{\hat \bfY}: \bfZ \to \hat\bfY$ such that $\bfZ$ is relatively Bernoulli over $p_{\hat \bfX}$ and relatively Bernoulli over $p_{\hat \bfY}$. But then, Lemma \ref{lem:factors_rel_B_are_B} tells us that the factor $\s$-algebra $\s(p_{\hat \bfX} \vee p_{ \hat \bfY})$ is relatively very weak Bernoulli over $p_{\hat \bfX}$ and relatively very weak Bernoulli over $p_{\hat \bfY}$. Therefore, there exist a Bernoulli shift $\tilde \bfB$ and two factor maps $\phi_1: \bfZ \to \tilde\bfB$ and $\phi_2: \bfZ \to \tilde\bfB$ such that $\phi_1 \indep p_{\hat \bfX}$, $\phi_2 \indep p_{\hat \bfY}$ and 
$$\s(p_{\hat \bfX} \vee \phi_1) = \s(p_{\hat \bfX} \vee p_{\hat \bfY}) = \s(p_{\hat \bfY} \vee \phi_2).$$
This implies that 
$$\hat\bfX \otimes \tilde \bfB \cong \hat\bfY \otimes \tilde\bfB.$$
But, since we chose to have $h(\hat\bfX) = h(\hat\bfY) \leq a/3$, we get 
\begin{align*}
h(\tilde\bfB) \leq h(p_{\hat\bfX} \vee p_{\hat\bfY}) \leq h(\hat \bfX) + h(\hat \bfY) \leq 2a/3 \leq h(\hat \bfB).
\end{align*}
Given a last Bernoulli shift $\overline{\bfB}$ of entropy $h(\hat\bfB) - h(\tilde\bfB)$ we get $\hat\bfB \cong \tilde \bfB \otimes \overline{\bfB}$ and
\begin{align*}
\bfX \cong \hat\bfX \otimes \hat \bfB \cong \hat\bfX \otimes \tilde \bfB \otimes \overline{\bfB} \cong \hat\bfY \otimes \tilde\bfB \otimes \overline{\bfB} \cong \hat\bfY \otimes \hat \bfB \cong \bfY.
\end{align*}
\end{proof}

As a consequence of this result, we see that if $\F := (\F_n)_{n \leq 0}$ and $\G := (\G_n)_{n \leq 0}$ are two weak Pinsker filtrations on $\bfX$ such that, for all $n \leq 0$, $h_\mu(\F_n, T) = h_\mu(\G_n, T)$, then we must have that, for each $n \leq 0$, $\quotient{\bfX}{\F_n} \cong \quotient{\bfX}{\G_n}$. 

However, this only gives ``local isomorphisms'', and it does not necessarily mean that the filtrations $\F$ and $\G$ are isomorphic (according to the notion of isomorphism introduced in Definition \ref{def:iso_filtration_dyn}). Therefore, the following is still an open question:
\begin{quest}
Let $\bfX := (X, \A, \mu, T)$ be an ergodic dynamical system. Are all weak Pinsker filtrations on $\bfX$ with the same entropy isomorphic ?
\end{quest}
\noindent This question is what we call the \emph{uniqueness problem}. 

If $\bfX$ is a Bernoulli shift, and if we take an increasing sequence $(\eps_n)_{n \leq 0}$ such that $\eps_0 = h(\bfX)$, we can take Bernoulli shifts $(\bfB_n)_{n \leq 0}$ such that $h(\bfB_n) = \eps_n - \eps_{n-1}$, and define the system 
$$\bfB := \bigotimes_{n \leq 0} \bfB_n.$$
It is a Bernoulli shift of entropy $\eps_0 = h(\bfX)$, so it is isomorphic to $\bfX$. Through this isomorphism, the factors of the form $\bigotimes_{k \leq n} \bfB_k$ generate a product type weak Pinsker filtration on $\bfX$. Therefore, in the case where $\bfX$ is a Bernoulli shift, the uniqueness problem becomes:
\begin{quest} \label{quest:uniqu_problem_bernoulli}
Let $\bfX := (X, \A, \mu, T)$ be a Bernoulli shift. Are all weak Pinsker filtrations on $\bfX$ of product type ?
\end{quest}

\section{Uniqueness problem on Bernoulli systems}
\label{sect:unique_prob_bernoulli}

In this section, we present our efforts to tackle Question \ref{quest:uniqu_problem_bernoulli}. The ideas developed here come from discussions with Jean-Paul Thouvenot, and we thank him for those insights. 
 Specifically, we are going to show:

\begin{thm} \label{thm:unique_weak_pinsker_bernoulli}
Let $\bfX := (X, \A, \mu, T)$ be a Bernoulli system and let $\F := (\F_n)_{n \leq 0}$ be a weak Pinsker filtration. There exists some sub-sequence $(\F_{n_k})_{k \leq 0}$ which is a weak Pinsker filtration of product type.
\end{thm}
\noindent The fact that we are only able to describe the structure of a sub-sequence of $\F$, for now, seems to be significant. Indeed, we can compare that result to a well known result from Vershik about static filtrations on a probability space: any filtration whose tail $\s$-algebra $\bigcap_{n \leq 0} \F_n$ is trivial has a sub-sequence that is standard (see \cite[Theorem 3]{Em_schach}). However there are many examples of non-standard filtrations with trivial tail $\s$-algebra. Therefore, although the context of Vershik's result is very different, it emphasizes that Theorem \ref{thm:unique_weak_pinsker_bernoulli} does not give a complete answer to Question \ref{quest:uniqu_problem_bernoulli}.

The main step in proving Theorem \ref{thm:unique_weak_pinsker_bernoulli} is contained in the following proposition:
\begin{prop} \label{prop:unique_weak_pinsker_bernoulli}
Let $\bfX := (X, \A, \mu, T)$ be a Bernoulli system of finite entropy and $\calP_0: X \arr A$ a finite generator of $\A$, i.e. a finite valued random variable such that $\A = \s(\{\calP_0 \circ T^n\}_{n \in \bbZ})$. For every $\eps > 0$, there exists $\delta > 0$ such that, if $\scrH \subset \A$ is a factor $\s$-algebra such that $\bfX$ is relatively Bernoulli over $\scrH$, and if $h_\mu(\scrH, T) \leq \delta$, there is a Bernoulli factor $\s$-algebra $\B$ such that 
\begin{enumerate} [label = (\roman*)]
\item $\B \indep \scrH$, 
\item $\A = \scrH \vee \B$ mod $\mu$,
\item and $\calP_0 \preceq_\eps \B$. 
\end{enumerate}
\end{prop}
In this proposition, Krieger' theorem (Theorem \ref{thm:krieger}) ensures the existence of a finite generator $\calP$ since $\bfX$ has finite entropy. The notation ``$\calP_0 \preceq_\eps \B$'', which we use many times below, means that there exists a $\B$-measurable random variable $\Q_0$ such that $\mu(\calP_0 \neq \Q_0) \leq \eps$.

The existence of a Bernoulli factor satisfying (i) and (ii) is simply the definition of relative Bernoullicity, the important part of this proposition is the ability to build a Bernoulli complement that satisfies (iii). Then iterating this result will yield Theorem \ref{thm:unique_weak_pinsker_bernoulli} (see Section \ref{sect:proof_of_thm_331}). 

\subsection{The technical lemma}

In this section, we tackle the main technical and constructive part of the proof of Proposition \ref{prop:unique_weak_pinsker_bernoulli}. It is contained in Lemma \ref{lem:unique_weak_pinsker_bernoulli}.

In Section \ref{sect:entropy_ornstein}, we introduced the notion of very weak Bernoullicity, which gives a characterization of Bernoulli systems. Here, we use another equivalent notion: extremality, due to Thouvenot (see \cite[Definition 6.3]{ref_extremalite_better}). 
\begin{defi}
Let $\bfX := (X, \A, \mu, T)$ be an ergodic dynamical system and $\xi := (\xi_0 \circ T^n)_{n \in \bbZ}$ be a process where $\xi_0$ takes values in some finite alphabet $A$. We say that $\xi$ is \emph{extremal} if, for every $\eps >0 $, there exist $\delta >0$ and $N \in \bbN$, such that for every $\ell \geq N$ and every random variable $\Q: X \arr B$ with $\#B \leq 2^{\delta \ell}$, there is a set $B_0 \subset B$ such that $\mu(\Q \in B_0) \geq 1- \eps$ and for $b \in B_0$, we have:
$$\bar{d_\ell}(\nu_\ell(\cdot \, | \, b), \nu_\ell(\cdot)) \leq \eps,$$
where $\nu_\ell$ is the law of $\xi_{[0, \ell[}$ and $\nu_\ell(\cdot \, | \, b)$ is the law of $\xi_{[0, \ell[}$ given that $\Q$ equals $b$.
\end{defi}
In \cite[Theorem 6.4]{ref_extremalite_better}, it is shown that extremality is equivalent to very weak Bernoullicity (and hence to Bernoullicity). In particular, we will use the fact that any process defined on a Bernoulli system is extremal. 

The proof of Lemma \ref{lem:unique_weak_pinsker_bernoulli} uses many methods that are usual in Ornstein's theory of Bernoulli shifts (a presentation can be found in \cite{book_ornstein} or \cite{theory_of_bernoulli_shifts}). Therefore, we need to introduce some commonly used notions and results from that theory. The following combinatorial result is frequently used in Ornstein's theory:
\begin{lem}[Hall's marriage lemma \cite{Hall_lemma_book}] \label{lem:hall}
Let $E$ and $F$ be finite sets, and $\{J_e\}_{e \in E}$ be a family of subsets of $F$: $\forall e \in E, J_e \subset F$. There exists an injective map $\psi: E \arr F$ such that $\forall e \in E, \psi(e) \in J_e$ if, and only if for every $I \subset E$, we have 
$$\#I \leq \#\bigcup_{e \in I} J_e.$$
\end{lem}

The main way in which the entropy of the processes is used in our arguments comes from the Shannon-McMillan-Breiman Theorem (see \cite[Theorem 13.1]{book_billingsley}):
\begin{thm} \label{thm:SMB}
Let $\bfX := (X, \A, \mu, T)$ be an ergodic dynamical system and $\xi_0: X \arr A$. For $\bf{a} \in A^{[0, n[}$, define
$$p_n(\bf{a}) := \mu(\xi_{[0, n[} = \bf{a}).$$
We have
$$\lim_{n \arr \infty} - \frac{1}{n} \log(p_n(\xi_{[0, n[})) = h_\mu(\xi, T), \; \mu \text{-almost surely}.$$
In particular, we also have the convergence in probability: for every $\eps > 0$, there exists $N \geq 1$ such that for every $n \geq N$, there exists a set $\calA_n \subset A^{[0, n[}$ such that $\mu(\xi_{[0, n[} \in \calA_n)\geq 1-\eps$ and for every $\bf{a} \in \calA_n$, 
$$2^{-(h_\mu(\xi, T) + \eps)n} \leq \mu(\xi_{[0, n[} = \bf{a}) \leq 2^{-(h_\mu(\xi, T) -\eps)n}.$$
\end{thm}

We also need to introduce another tool that is commonly used in Ornstein's theory: Rokhlin towers. On a dynamical system $\bfX := (X, \A, \mu, T)$, to get a tower of height $n$, we need a set $F$ such that the sets $T^jF$, for $0 \leq j \leq n-1$ are disjoint. Then the family $\calT := (F, TF, ..., T^{n-1}F)$ is what we call a Rokhlin tower, or, in short, a \emph{tower}. However, we will also refer to the set $\bigsqcup_{j=0}^{n-1} T^jF$ as a tower. In particular, many times, we will write $\mu(\calT)$ for $\mu(\bigsqcup_{j=0}^{n-1} T^jF)$. The following result guaranties that Rokhlin of arbitrary height and total measure almost $1$ exist under quite general conditions:
\begin{prop}[See \cite{theory_of_bernoulli_shifts}] \label{thm:rokhlin}
Let $\bfX := (X, \A, \mu, T)$ be an ergodic dynamical system and $\xi_0$ a finite valued random variable. Assume that $\mu$ is non-atomic. For all $n \geq 1$ and $\eps > 0$, there exists a measurable set $F \subset X$ such that the sets $T^jF$, for $j \in [0, n[$, are disjoint, $\mu(\bigcup_{j=0}^{n-1} T^jF) \geq 1- \eps$ and $\calL(\xi_0 \, | \, F) = \calL(\xi_0)$. 
\end{prop}
\noindent The set $F$ is called the base of the tower $\calT$ and the sets $T^jF$ are the levels. For any set $E \subset F$, the family
$$C_E := \{T^jE\}_{0 \leq j \leq n-1}$$
is a tower, and we say that it is a column of $\calT$. If $\xi_0: X \arr A$ is a random variable, we will be interested in the columns defined by sets of the form $F_{\bf{a}} := F \cap \{\xi_{[0, n[}= \bf{a}\}$ with $\bf{a} \in A^{[0, n[}$. We say that $\bf{a}$ is the $\xi$-name of the column $C_{\bf{a}} := C_{F_{\bf{a}}}$. The columns $\{C_{\bf{a}}\}_{\bf{a} \in A^{[0, n[}}$ give a partition of the levels of $\calT$. Now, conversely, assume that we have a partition of $F$ given by sets $E_1, ..., E_p$, then the columns $C_{E_1}, ..., C_{E_p}$ give a partition of the levels of $\calT$. If, moreover, we associate to each column $C_{E_i}$ a name $\bf{a}^{(i)} \in A^{[0, n[}$ of length $n$, we can define a random variable $\xi_0$ on the levels of $\calT$ so that, for every $i$, we have $C_{E_i} = C_{\bf{a}^{(i)}}$. We obtain this random variable simply by setting, for $i \in \llbracket 1, p \rrbracket, j \in \llbracket 0, n \llbracket$
$$\xi_0 = \bf{a}_j^{(i)} \; \text{ on } \; T^jE_i.$$
This is the framework we will use to construct our random variables. We are now ready to turn our attention to the following:

\begin{lem} \label{lem:unique_weak_pinsker_bernoulli}
Let $\bfX := (X, \A, \mu, T)$ be a Bernoulli system of finite entropy and $\calP_0: X \arr A$ a finite generator of $\A$. For every $\eps > 0$, there exists $\delta > 0$ satisfying the following:
\begin{itemize}
\item if $\calH_0 : X \arr H$ is a finite valued random variable such that $h_\mu(\calH, T) \leq \delta$,
\item and $\xi := (\xi_0 \circ T^n)_{n \in \bbZ}$ is a $B$-valued (for some finite set $B$) i.i.d. process independent from $\calH$ such that $\A = \s(\calH) \vee \s(\xi)$ mod $\mu$,
\end{itemize}
then for any $\a >0$, there exists a process $\tilde\xi$ such that
\begin{enumerate} [label = (\roman*)]
\item $\bar{d}_1(\calL(\xi_0), \calL(\tilde\xi_0)) \leq \a$,
\item $0 \leq h_\mu(\calH \vee \xi, T) - h_\mu(\calH \vee \tilde \xi, T) \leq \a$,
\item and $\calP_0 \preceq_{\eps} \s(\tilde\xi)$.
\end{enumerate}
\end{lem}
The proof of the lemma being quite intricate, we start by giving a sketch of the proof. First, we will need a Rokhlin tower $\calT_n$ of very large height $n$. This tower is then divided into the columns $C_{\bf{h}}$ (see \eqref{eq:def_H_column}) generated by $\calH$. Each of those columns is then divided into sub-columns $C_{\bf{h}}^{\bf{b}}$ (see \eqref{eq:def_xi_column}) generated by $\xi$. Because $\calH \vee \xi$ generates $\A$, we can approach $\calP_0$ by some random variable $\tilde\calP_0$ depending on finitely many coordinates of $\calH \vee \xi$. It enables us to associate to each $C_{\bf{h}}^{\bf{b}}$ a word $\tilde\calP_{[0, n[}(\bf{h}, \bf{b})$ which gives a good approximation of $\calP_0$ on the levels of $C_{\bf{h}}^{\bf{b}}$. We will define $\tilde\xi_0$ by giving $C_{\bf{h}}^{\bf{b}}$ a new $\tilde\xi$-name, to replace $\bf{b}$. Our goal is to choose those names so that we can get a good approximation of $\tilde\calP_{[0, n[}(\bf{h}, \bf{b})$ by simply knowing the $\tilde\xi$-name of $C_{\bf{h}}^{\bf{b}}$, regardless of $\bf{h}$. To do that, we fix a column $C_{\bf{h}_0}$ and use it as a ``model'' for the other columns. Then the extremality of $\calP$ comes into play: it tells us, for most choices of $\bf{h}$, the families $\{\tilde\calP_{[0, n[}(\bf{h}_0, \bf{b})\}_{\bf{b} \in \calB_n}$ and $\{\tilde\calP_{[0, n[}(\bf{h}, \bf{b})\}_{\bf{b} \in \calB_n}$ are quite similar. More specifically, we show that, for most $\bf{b}$, there are names $\tilde{\bf{b}}$ such that ${d_n(\tilde\calP_{[0, n[}(\bf{h}_0, \tilde{\bf{b}}), \tilde\calP_{[0, n[}(\bf{h}, \bf{b}))}$ is small. Those names are then suitable $\tilde\xi$-names for $C^{\bf{b}}_{\bf{h}}$. However, when we choose among those suitable names, we need to make sure that we are not giving the same name to too many columns, otherwise we might loose to much information, and we could not get (ii). This is done using Hall's marriage lemma.

\begin{proof}[Proof of Lemma \ref{lem:unique_weak_pinsker_bernoulli}]
In this proof, we use many parameters, which we introduce below in a specific order to highlight the way they depend on each other:
\begin{enumerate} [label = (\alph*)]
\item \label{a} Let $\eps > 0$. This parameter is chosen first, as it appears in the statement of the lemma. Then we choose $\delta > 0$ and $N \geq 1$, as the numbers associated to $\eps^3/4$ in the definition of extremality of $\calP$. We assume that $h_\mu(\calH, T) < \delta$.
\item \label{b} Let $\alpha > 0$. This is another arbitrarily small parameter that appears in the statement of the lemma. It does not depend on $\eps$ nor $\delta$. 
\item \label{c} Next, we introduce $0 < \g < 1$, which must be small relative to $\a$ and $\eps$ for (ii) and (iii) to hold. Specifically, we require that $\g \leq  \eps/2$, and that the bound in Lemma \ref{lem:continuity_entropy} holds with error $\alpha$, whenever $\mu(\xi_0 \neq \zeta_0) \leq 2\g$, for any random variables $\xi_0$ and $\zeta_0$.
\item \label{d} Then we take $\beta > 0$, which is our most used parameter. We set $\beta$ satisfying the following:
$$\beta < \left\{ 
\begin{array}{l}
	\delta - h_\mu(\calH, T); \\
	\min(\eps^3/24, \eps/4); \\
	\min(\g, \g h_\mu(\xi, T)/5); \\
	\a/7.
\end{array}
\right.$$
Once $\beta$ is fixed, we choose $n_0 \geq 1$ such that $\calP_0 \preceq_{\beta^2/2} (\calH \vee \xi)_{[-n_0, n_0]}$.
\item \label{e} Finally, we choose an integer $n$, which will be the height of the Rokhlin tower. It is chosen larger than $N$. We also need it to be large enough for us to apply the Shannon-McMillan-Breiman theorem, as well as Birkhoff's ergodic theorem. As $n$ appears in many estimates, it needs to be large enough depending on $\eps$, $\delta$, $\g$, $\beta$ and $n_0$. It would be quite tedious to give an explicit lower bound for $n$, so, since all other parameters are now fixed and do not depend on $n$, we simply point out throughout the proof the estimates where $n$ needs to be large. 
\end{enumerate}
Having now established the parameters, we begin the proof.

\subsubsection*{Step 1: The setup of the tower}
As mentioned in \ref{e}, we choose $n$ so that we can apply the Shannon-McMillan-Breiman theorem (i.e. Theorem \ref{thm:SMB}) and 
Birkhoff's ergodic theorem to know that there exist two sets $\calE_n^0 \subset H^{[0, n[}$ and $\calB_n^0 \subset B^{[0, n[}$ such that
\begin{eq} 
\begin{gathered}
\mu(\calH_{[0, n[} \in \calE_n^0) \geq 1-\beta/2, \; \text{ and } \; \mu(\xi_{[0, n[} \in \calB_n^0) \geq 1- \beta/3,
\end{gathered}
\end{eq}
on which the estimates \eqref{eq:estimate_H}, \eqref{eq:estimate_Hvxi}, \eqref{eq:law_ergo_thm} and \eqref{eq:estimate_tronc_xi} hold. Latter in the proof, we will see that we can take $\calE_n \subset \calE_n^0$ and $\calB_n \subset \calB_n^0$ subsets such that
\begin{eq} \label{eq:sets_good_estimates}
\begin{gathered}
\mu(\calH_{[0, n[} \in \calE_n) \geq 1-\beta, \; \text{ and } \; \mu(\xi_{[0, n[} \in \calB_n) \geq 1- \beta,
\end{gathered}
\end{eq}
on which we also have \eqref{eq:approx_calP} and \eqref{eq:tronc_column_full}. The fact that \eqref{eq:approx_calP} holds for $\calE_n$ appears in Step 2 and the fact that \eqref{eq:tronc_column_full} holds for $\calB_n$ appears in Step 3. Until then, we only use \eqref{eq:sets_good_estimates}. For now, we present some of the estimates we have announced. 

The first estimates given by the Shannon-McMillan-Breiman theorem are:
\begin{align}
&\forall \bf{h} \in \calE_n, \; 2^{-(h_\mu(\calH, T) + \beta)n} \leq \mu(\calH_{[0, n[} = \bf{h}) \leq 2^{-(h_\mu(\calH, T) - \beta)n}, \label{eq:estimate_H}\\
&\forall \bf{b} \in \calB_n, \; 2^{-(h_\mu(\xi, T) + \beta)n} \leq \mu(\xi_{[0, n[} = \bf{b}) \leq 2^{-(h_\mu(\xi, T) - \beta)n}. \label{eq:estimate_Hvxi}
\end{align}

For any sequence $\bf{b} \in B^{[0, n[}$ and any element $b' \in B$, denote $f_n(\bf{b}, b')$ the frequency at which the element $b'$ appears in the sequence $\bf{b}$. This can  also be defined as follows:
\begin{eq} \label{eq:def_frequency}
\forall x \in \{\xi_{[0, n[} = \bf{b}\}, \; f_n(\bf{b}, b') := \frac{1}{n } \sum_{j = 0}^{n-1}\mathbbm{1}_{\{\xi_{0} = b'\}}(T^jx).
\end{eq}
From this definition of $f_n$, it becomes clear that, as announced earlier, the estimate given by Birkhoff's ergodic theorem is:
\begin{eq} \label{eq:law_ergo_thm}
\sum_{b' \in B} |f_n(\bf{b}, b') - \mu(\xi_{0} = b')| \leq \beta.
\end{eq}
Since $\calH \vee \xi$ generates $\A$, as said in \ref{d}, we can find $n_0 \geq 1$ so that $\calP_0 \preceq_{\beta^2} (\calH \vee \xi)_{[-n_0, n_0]}$. This means that there exists a $(\calH \vee \xi)_{[-n_0, n_0]}$-measurable random variable $\tilde\calP_0$ such that $\mu(\tilde\calP_0 \neq \calP_0) \leq \beta^2$. 

By making use of Proposition \ref{thm:rokhlin}, we can build a set $G$ such that $F' := TG$ is disjoint from $G$ and $F'$ is the base of a tower $\calG_n := \{T^jF'\}_{0 \leq j \leq n-1}$ such that $\mu(\calG_n) \geq 1 - \beta$ and 
\begin{eq} \label{eq:indep_base_tower}
\calL((\calH \vee \tilde\calP \vee \xi)_{[0, n[} \, | \, F') =  \calL((\calH \vee \tilde\calP \vee \xi)_{[0, n[}).
\end{eq}
The set $G$ will be useful later to code the entrance of the tower. We slightly reduce the tower by setting $F := F' \cap \{ \calH_{[0, n[} \in \calE_n\} \cap \{\xi_{[0, n[} \in \calB_n\}$ and $\calT_n := \{T^jF\}_{0 \leq j \leq n-1}$. One can then use \eqref{eq:indep_base_tower} with our previous estimates to see that $\mu(\calT_n) \geq 1-3\beta$ (by making sure that $1/n \leq \beta$).


We then split $\calT_n$ into $\calH$-columns: for $\bf{h} \in {\calE}_n$, we define 
\begin{eq} \label{eq:def_H_column}
C_{\bf{h}} := \{T^j(F \cap \{\calH_{[0, n[} = \bf{h}\})\}_{0 \leq j \leq n-1},
\end{eq}
so that $\calT_n = \bigsqcup_{\bf{h} \in {\calE}_n} C_{\bf{h}}$ (we mean that the levels of $\calT_n$ are disjoint unions of the levels of $C_{\bf{h}}$). For each $\bf{h} \in \calE_n$, we say that $C_{\bf{h}}$ is the column of $\calH_{[0, n[}$-name $\bf{h}$. We also denote by $F_{\bf{h}} := F \cap \{\calH_{[0, n[} = \bf{h}\}$ the base of $C_{\bf{h}}$.

\subsubsection*{Step 2: Using the extremality of $\calP$}
We plan on modifying $\xi$ into a process $\tilde\xi$ so that the joint law of $\calP \vee \tilde\xi$ is almost the same in most of the columns $\{C_{\bf{h}}\}_{\bf{h} \in \calE_n}$. We start by using the fact that $\bfX$ is Bernoulli to see that the law of $\calP$ is almost the same on each column $C_{\bf{h}}$. Indeed, since $\bfX$ is Bernoulli, $\calP$ is extremal, and we fixed $\delta > 0$ and $N \geq 1$ as the numbers associated to $\eps^3/4$ in the definition of extremality and assume that $h_\mu(\calH, T) < \delta$ (see \ref{a}). On the other hand, from \eqref{eq:estimate_H}, we deduce that 
$$\# \calE_n \leq 2^{(h_\mu(\calH, T) + \beta)n}.$$
Next we define the partition
$$\Q := \left\{
\begin{array}{ll}
* & \text{on } \{\calH_{[0, n[} \notin \calE_n\} \cup \{\xi_{[0, n[} \notin \calB_n\}\\
\calH_{[0, n[}  & \text{on } \{\calH_{[0, n[} \in \calE_n\} \cap \{ \xi_{[0, n[} \in \calB_n\}\\
\end{array}.
\right.$$
In particular, we know that $\mu(\Q = *) \leq 2 \beta$. Moreover, the number of values taken by $\Q$ is bounded by 
$$\# \calE_n + 1 \leq 2^{(h_\mu(\calH, T) + \beta)n} +1 \leq 2^{n\delta},$$
since $\beta < \delta - h_\mu(\calH, T)$. Therefore the extremality of $\calP$ tells us that, since $n \geq N$, there exists a subset $\bar{\calE}_n \subset {\calE}_n$ such that 
\begin{eq} \label{eq:set_extremality}
\mu(\Q \notin (\bar{\calE}_n \cup \{*\})) \leq \eps^3/4 + 2\beta \leq \eps^3 \leq \eps,
\end{eq}
and for $\bf{h} \in \bar{\calE}_n$, we have
$$\bar{d}_n(\calL(\calP_{[0, n[} \, | \, \Q = \bf{h}), \calL(\calP_{[0, n[})) \leq \eps^3/4.$$
As mentioned at the start of the proof, the set $\calE_n$ is chosen so that \eqref{eq:sets_good_estimates} holds and we have
\begin{eq}  \label{eq:approx_calP}
\forall \bf{h} \in \calE_n, \; \mu(\tilde\calP_0 \neq \calP_0 \, | \, \Q = \bf{h}) \leq \beta.
\end{eq}
This is possible because $\mu(\calP_0 \neq \tilde\calP_0) \leq \beta^2/2$.
Therefore, using that $\beta \leq \eps^3/24$ with \eqref{eq:approx_calP}, this yields, for $\bf{h} \in \bar{\calE}_n$:
\begin{eq} \label{eq:good_approx_extremal}
\bar{d}_n(\calL(\tilde\calP_{[0, n[} \, | \, \Q = \bf{h}), \calL(\tilde\calP_{[0, n[})) \leq \eps^3/3.
\end{eq}

\subsubsection*{Step 3: Framework for the construction of $\tilde\xi_0$}
We start the construction of $\tilde\xi$ by setting $\tilde\xi_0 := *$ on $G = T^{-1}F'$, where $*$ represents a symbol that does not belong to $B$. Later in the proof, this will allow us to detect the entrance into $\calT_n$ from the value of the process $\tilde\xi$. Then define $\tilde \xi_0$ to take any value in $B$ on the rest of $\calT_n^c$. For $\bf{h} \in \calE_n \backslash \bar{\calE}_n$, on $C_{\bf{h}}$, we set $\tilde\xi_0 := \xi_0$. We are left with defining our new random variable $\tilde\xi_0$ on the columns $C_{\bf{h}}$, with $\bf{h} \in \bar{\calE}_n$. We start by fixing $\bf{h}_0 \in \bar{\calE}_n$, and the column $C_{\bf{h}_0}$ will serve as a ``model'' for the other columns. 

Next we fix an $\bf{h} \in \bar{\calE}_n$. We define sub-columns of $C_{\bf{h}}$: for $\bf{b} \in \calB_n$, 
\begin{eq} \label{eq:def_xi_column}
C_{\bf{h}}^{\bf{b}} := \{T^j(F \cap \{\calH_{[0, n[} = \bf{h}\} \cap \{\xi_{[0, n[} = \bf{b}\})\}_{0 \leq j \leq n-1}.
\end{eq}
We say that $\bf{b}$ the $\xi$-name of $C_{\bf{h}}^{\bf{b}}$. Because of our definition of $F$ and \eqref{eq:indep_base_tower}, the set $\calB_n$ gives us exactly the $\xi$-names of all the sub-columns in $C_{\bf{h}}$. We will then give each sub-column $C_{\bf{h}}^{\bf{b}}$ a new word $\tilde{\bf{b}} \in \calB_n$ and define the random variable $\tilde\xi_0$ on $C_{\bf{h}}^{\bf{b}}$ as the only variable such that $\tilde{\bf{b}}$ is the $\tilde\xi$-name of $C_{\bf{h}}^{\bf{b}}$. This means that to conclude the construction of $\tilde\xi_0$ on $C_{\bf{h}}$, we simply need to build a map $\phi_{\bf{h}}: \calB_n \to \calB_n,$ and the properties we will obtain on $\tilde \xi$ will follow from our choice for $\phi_{\bf{h}}$. 

In order to give us some additional leeway, we use the parameter $\gamma>0$ introduced at the start of the proof: we define $n_1 := \lfloor (1-\g)n\rfloor \leq n$, and for $\bf{b} \in \calB_n$, we denote by $\bf{b}_{n_1} := \bf{b}_{[0, n_1[} \in B^{[0, n_1[}$ the truncated sub-sequence of $\bf{b}$ of length $n_1$. Conversely, for $\bar{\bf{b}} \in B^{[0, n_1[}$, define 
$$B(\bar{\bf{b}}) := \{\bf{b} \in \calB_n \, | \, \bf{b}_{n_1} = \bar{\bf{b}}\},$$
and 
$$\calB_{n_1} := \{\bar{\bf{b}}  \in B^{[0, n_1[} \, | \, B(\bar{\bf{b}}) \neq \varnothing\}.$$
We will obtain the map $\phi_{\bf{h}}$ by building an injective map $\psi_{\bf{h}}: \calB_{n_1} \to \calB_n$ and setting $\phi_{\bf{h}}(\bf{b}) := \psi_{\bf{h}}(\bf{b}_{n_1})$.  We start be recalling that we chose $n$ and $\calB_n$ so that the estimate given by the Shannon-McMillan-Breiman Theorem, i.e. \eqref{eq:estimate_Hvxi}, still holds when replacing $n$ by $n_1$. More precisely, we mean that, for $\bar{\bf{b}} \in \calB_{n_1}$
\begin{eq} \label{eq:estimate_tronc_xi}
2^{-(h_\mu(\xi, T) + \beta)n_1} \leq \mu(\xi_{[0, n_1[} = \bar{\bf{b}}) \leq 2^{-(h_\mu(\xi, T) - \beta)n_1}.
\end{eq}
Moreover, we stated at the start of the proof that $\calB_n$ is chosen such that, for $\bar{\bf{b}} \in \calB_{n_1}$
\begin{eq} \label{eq:tronc_column_full}
\mu(\xi_{[0, n_1[} = \bar{\bf{b}}, \xi_{[0, n[} \in \calB_n) \geq \frac{1}{2} \mu(\xi_{[0, n_1[} = \bar{\bf{b}}).
\end{eq}
We need to prove that statement. We do this by considering the set 
$$\calC := \{\bar{\bf{b}} \in \calB_{n_1}^0 \, | \, \mu(\xi_{[0, n_1[} = \bar{\bf{b}}, \xi_{[0, n[} \notin \calB_n^0) \geq \frac{1}{2} \mu(\xi_{[0, n_1[} = \bar{\bf{b}})\}.$$
From the definition of $\calC$, we get 
$$\frac{1}{2}\mu(\xi_{[0, n_1[} \in \calC) \leq \mu(\xi_{[0, n[} \notin \calB_n^0) \leq \beta/3.$$
Then, we define $\calB_n$ as $\calB_n := \{\bf{b} \in \calB_n^0 \, | \, \bf{b}_{n_1} \notin \calC\}$, and easily get that $\mu(\xi_{[0, n[} \in \calB_n) \geq 1 - \beta$. Next, because the set we removed from $\calB_n^0$ is measurable with respect to the truncated sequences of length $n_1$, for $\bar{\bf{b}} \notin \calC$, we get 
$$\mu(\xi_{[0, n_1[} = \bar{\bf{b}}, \xi_{[0, n[} \in \calB_n) = \mu(\xi_{[0, n_1[} = \bar{\bf{b}}, \xi_{[0, n[} \in \calB_n^0),$$
so \eqref{eq:tronc_column_full} follows from the definition of $\calC$. 

Finally, putting \eqref{eq:estimate_tronc_xi} and \eqref{eq:tronc_column_full} together, we get, for $\bar{\bf{b}} \in \calB_{n_1}$
\begin{eq} \label{eq:another_estimate}
\mu(\xi_{[0, n_1[} = \bar{\bf{b}}, \xi_{[0, n[} \in \calB_n) \geq \frac{1}{2} 2^{-(h_\mu(\xi, T) + \beta)n_1} \geq 2^{-(h_\mu(\xi, T) + 2\beta)n_1},
\end{eq}
by making sure that $n_1$ large enough. This will enable us to control the measure of the part of the truncated column over $\{\xi_{[0, n_1[} = \bar{\bf{b}}\}$ that is in $\calT_n$.

\subsubsection*{Step 4: Estimates for Hall's marriage lemma}
From \eqref{eq:estimate_Hvxi} and \eqref{eq:estimate_tronc_xi}, we can tell that
\begin{align*}
\#\calB_{n_1} &\leq 2^{(h_\mu(\xi, T) + \beta)n_1} \leq 2^{(h_\mu(\xi, T) + \beta)(1-\g)n} \\
& \leq 2^{(h_\mu(\xi, T) - \g h_\mu(\xi, T) + \beta)n} \leq (1-\beta)2^{(h_\mu(\xi, T) - \beta)n} \leq \#\calB_n,
\end{align*}
since $2 \beta < \g h_\mu(\xi, T)$ and we can choose $n$ large enough. This inequality is also clearly true from the definition of $\calB_{n_1}$, but we include this computation, as a similar one will be essential later in the proof. That being said, this inequality means that it is possible to find an injective map from $\calB_{n_1}$ to $\calB_n$, but we want to be more specific about which injective map we choose. To that end, we will make use of Hall's marriage lemma. To do that, for each $\bar{\bf{b}} \in \calB_{n_1}$, we need to specify which elements of $\calB_n$ we consider suitable $\tilde\xi$-names for the columns $\{C_{\bf{h}}^{\bf{b}} ; \, \bf{b} \in B(\bar{\bf{b}})\}$. 

We recall that $n_0$ is the integer chosen so that $\tilde\calP_0$ is $(\calH \vee \xi)_{[-n_0, n_0]}$-measurable. Define $L_n := [n_0, n_1-n_0[ \subset \bbZ$ and $\ell := n_1-2n_0$ the length of $L_n$. Because $\tilde\calP_0$ is $(\calH \vee \xi)_{[-n_0, n_0]}$-measurable, $\tilde\calP_{L_n}$ is $(\calH \vee \xi)_{[0, n_1[}$-measurable. So, for $\bf{h}$ fixed, for each $\bar{\bf{b}} \in \calB_{n_1}$, on the set $\{\calH_{[0, n[} = \bf{h}, \xi_{[0, n_1[} = \bar{\bf{b}}\}$, there can be only one value of $\tilde\calP_{L_n}$, which we denote $\tilde\calP_{L_n}(\bf{h}, \bar{\bf{b}})$. 

For $\bar{\bf{b}} \in \calB_{n_1}$, the suitable corresponding $\tilde\xi$-names will be the elements $\bf{b} \in \calB_n$ for which $d_\ell(\tilde\calP_{L_n}(\bf{h}_0, \bf{b}_{n_1}), \tilde\calP_{L_n}(\bf{h}, \bar{\bf{b}})) \leq \eps$. More formally, we set
$$J_{\bar{\bf{b}}} := \{\bf{b} \in \calB_n \, | \, d_\ell(\tilde\calP_{L_n}(\bf{h}_0, \bf{b}_{n_1}), \tilde\calP_{L_n}(\bf{h}, \bar{\bf{b}})) \leq \eps\},$$
and we want to build $\psi_{\bf{h}}$ so that we have 
\begin{eq}  \label{eq:marriage_condition}
\psi_{\bf{h}}(\bar{\bf{b}}) \in J_{\bar{\bf{b}}},
\end{eq}
for as many $\bar{\bf{b}} \in \calB_{n_1}$ as possible. 

From \eqref{eq:good_approx_extremal}, it follows that 
$$\bar{d}_n(\calL(\tilde\calP_{[0, n[} \, | \, \Q= \bf{h}), \calL(\tilde\calP_{[0, n[} \, | \, \Q = \bf{h}_0)) \leq 2\eps^3/3.$$
Therefore:
\begin{align*}
\bar{d}_{\ell}(\calL(\tilde\calP_{L_n} \, | \, \Q= \bf{h})&, \calL(\tilde\calP_{L_n} \, | \, \Q = \bf{h}_0)) \\
&\leq \frac{n}{n_1-2n_0} \bar{d}_n(\calL(\tilde\calP_{[0, n[} \, | \, \Q= \bf{h}), \calL(\tilde\calP_{[0, n[} \, | \, \Q = \bf{h}_0)) \\
&\leq \frac{n}{(1-\g)n - n_0} 2\eps^3/3 < \eps^3,
\end{align*}
by choosing $n$ large enough. So there exists $\l \in \P(A^{L_n} \times A^{L_n})$ a coupling of $\calL(\tilde\calP_{L_n} \, | \, \Q= \bf{h})$ and $\calL(\tilde\calP_{L_n} \, | \, \Q = \bf{h}_0)$ such that 
$$\int d_\ell(\bf{p}_1, \bf{p}_2) d\l(\bf{p}_1, \bf{p}_2) \leq \eps^3.$$
Denote by $\l_1$ and $\l_2$ the marginals of $\l$, i.e. $\l_1 = \calL(\tilde\calP_{L_n} \, | \, \Q= \bf{h})$ and $\l_2 = \calL(\tilde\calP_{L_n} \, | \, \Q = \bf{h}_0)$. We are interested in the set $\calA_{\ell} \subset A^{L_n}$ defined by 
$$\calA_\ell := \{ \bf{p} \in A^{L_n} \, ; \, \l(d_\ell(\bf{p}_1, \bf{p}_2)\leq \eps \, | \, \bf{p}_1 = \bf{p}) \geq 1-\eps\}.$$
The following gives an estimate on the measure of $\calA_\ell$:
\begin{align*}
\eps^3 \geq& \int  d_\ell(\bf{p}_1, \bf{p}_2) d\l(\bf{p}_1, \bf{p}_2) \geq \int_{\bf{p}_1 \notin \calA_\ell}  d_\ell(\bf{p}_1, \bf{p}_2) d\l(\bf{p}_1, \bf{p}_2)\\
&= \int_{\bf{p} \notin \calA_\ell} \int d_\ell(\bf{p}_1, \bf{p}_2) d\l(\bf{p}_1, \bf{p}_2 \, | \, \bf{p}_1 = \bf{p}) d\l_1(\bf{p})\\
&\geq \int_{\bf{p} \notin \calA_\ell} \l(d_\ell(\bf{p}_1, \bf{p}_2) > \eps \, | \, \bf{p}_1 = \bf{p}) \cdot \eps d\l_1(\bf{p})\\
&> \eps^2 \mu(\tilde\calP_{L_n} \notin \calA_\ell \, | \, \Q = \bf{h}),
\end{align*}
so $\mu(\tilde\calP_{L_n} \notin \calA_\ell \, | \, \Q = \bf{h}) < \eps$. In other words, if we set 
$$\bar{\calB}_{n_1}(\bf{h}) := \{\bar{\bf{b}} \in \calB_{n_1} \, | \, \tilde\calP_{L_n}(\bf{h}, \bar{\bf{b}}) \in \calA_\ell\},$$
we have $\mu(\xi_{[0, n_1[} \in \bar{\calB}_{n_1}(\bf{h}) \, | \, \Q = \bf{h}) \geq 1-\eps$. The set $\bar{\calB}_{n_1}(\bf{h})$ is the set on which we want \eqref{eq:marriage_condition} to hold. Hall's marriage lemma tells us that there exists an injective map $\psi_{\bf{h}}: \bar{\calB}_{n_1}(\bf{h}) \arr \calB_n$ for which \eqref{eq:marriage_condition} is true if we have the following:
\begin{eq} \label{eq:statement_marriage_condition}
\forall I \subset \bar{\calB}_{n_1}(\bf{h}), \; \#I \leq \# \bigcup_{\bar{\bf{b}} \in I} J_{\bar{\bf{b}}}.
\end{eq}
Let $I \subset \bar{\calB}_{n_1}(\bf{h})$. Consider $K := \bigcup_{\bar{\bf{b}} \in I} \{\tilde\calP_{L_n}(\bf{h}, \bar{\bf{b}})\} \subset \calA_\ell$ and note that 
$$\bigcup_{\bar{\bf{b}} \in I} J_{\bar{\bf{b}}} = \{\bf{b} \in \calB_n \, | \, d_\ell(\tilde\calP_{L_n}(\bf{h}_0, \bf{b}_{n_1}), K) \leq \eps\}.$$
Taking that into account, we have 
\begin{eq} \label{eq:big_calcul}
\begin{split}
\# \bigcup_{\bar{\bf{b}} \in I} J_{\bar{\bf{b}}} &\geq 2^{(h_\mu(\xi, T) - \beta)n}\mu(d_\ell(\tilde\calP_{L_n}(\bf{h_0}, \xi_{[0, n_1[}), K) \leq \eps, \xi_{[0, n[} \in \calB_n)\\
&= 2^{(h_\mu(\xi, T) - \beta)n} \mu(d_\ell(\tilde\calP_{L_n}, K) \leq \eps, \xi_{[0, n[} \in \calB_n \, | \, \calH_{[0, n[} = \bf{h}_0)\\
&= 2^{(h_\mu(\xi, T) - \beta)n} \mu(\xi_{[0, n[} \in \calB_n) \mu(d_\ell(\tilde\calP_{L_n}, K) \leq \eps \, | \, \Q = \bf{h}_0)\\
&\geq 2^{(h_\mu(\xi, T) - \beta)n} (1 - \beta) \l(d_\ell(\bf{p}_1, \bf{p}_2) \leq \eps, \bf{p}_1 \in K)\\
&\geq 2^{(h_\mu(\xi, T) - 2\beta)n} \int_{\bf{p} \in K} \l(d_\ell(\bf{p}_1, \bf{p}_2) \leq \eps \, | \, \bf{p}_1 = \bf{p}) d\l_1(\bf{p}) \\
&\geq 2^{(h_\mu(\xi, T) - 2\beta)n} (1-\eps) \l_1(K), \, \text{ because } K \subset \calA_\ell\\
&\geq 2^{(h_\mu(\xi, T) - 3\beta)n} \mu(\tilde\calP_{L_n} \in K \, | \, \Q = \bf{h}).
\end{split}
\end{eq}
making sure again that $n$ is large enough. Moreover, using \eqref{eq:another_estimate}, we get
\begin{align*}
\# I &\leq 2^{(h_\mu(\xi, T) +2\beta)n_1} \mu(\xi_{[0, n_1[} \in I, \xi_{[0, n[} \in \calB_n)\\
&\leq 2^{(h_\mu(\xi, T) +2\beta)n_1} \mu(\xi_{[0, n_1[} \in I \, | \, \xi_{[0, n[} \in \calB_n)\\
&= 2^{(h_\mu(\xi, T) +2\beta)n_1} \mu(\xi_{[0, n_1[} \in I \, | \, \Q = \bf{h}), \, \text{ because } \xi \indep \calH\\
&\leq 2^{(h_\mu(\xi, T) +2\beta)n_1} \mu(\tilde\calP_{L_n} \in K \, | \, \Q = \bf{h}),\\
\end{align*}
by definition of $K$. Together with \eqref{eq:big_calcul}, it yields
\begin{align*}
\#I &\leq 2^{((1-\g)(h_\mu(\xi, T) + 2\beta) - (h_\mu(\xi, T)-3\beta))n} \# \bigcup_{\bar{\bf{b}} \in I} J_{\bar{\bf{b}}} \\
&\leq 2^{(5\beta-\g h_\mu(\xi, T))n} \# \bigcup_{\bar{\bf{b}} \in I} J_{\bar{\bf{b}}} \leq \# \bigcup_{\bar{\bf{b}} \in I} J_{\bar{\bf{b}}},
\end{align*}
since $5\beta \leq \g h_\mu(\xi, T)$. Therefore there exists an injective map $\psi_{\bf{h}}: \bar{\calB}_{n_1}(\bf{h}) \arr \calB_n$ for which \eqref{eq:marriage_condition} holds. As we noted that $\# \calB_{n_1} \leq \# \calB_n$, $\psi_{\bf{h}}$ can then be extended to an injective map defined on $\calB_{n_1}$ (still taking values in $\calB_n$). We recall that, with $\psi_{\bf{h}}$ built, we set $\phi_{\bf{h}}(\bf{b}) := \psi_{\bf{h}}(\bf{b}_{n_1})$. 

As we announced at the start of our reasoning, we define $\tilde \xi_0$ on the levels of $C_{\bf{h}}$ so that the $\tilde \xi$-name of each sub-column $C_{\bf{h}}^{\bf{b}}$ is $\phi_{\bf{h}}(\bf{b}) = \psi_{\bf{h}}(\bf{b}_{n_1})$. Since this construction can be done with every $\bf{h} \in \bar{\calE}_n$ (with the map $\psi_{\bf{h}}$ depending on $\bf{h}$), we have completed the construction of $\tilde\xi_0$. We now need to check that $\tilde\xi$ satisfies the conditions (i), (ii) and (iii) of our lemma.

\subsubsection*{Step 5: Proving that $\tilde\xi$ satsifies (i), (ii) and (iii)}
We start by estimating the law of $\tilde \xi_0$. Since $\mu(\calT_n) \geq 1-3\beta$, we have
\begin{align*}
\sum_{b \in B}  &|\mu(\tilde\xi_0 = b) - \mu(\xi_0 = b)| \leq  \sum_{b \in B} | \mu(\{\tilde\xi_0 = b \} \cap \calT_n) - \mu(\xi_0=b)\mu(\calT_n)| + 6 \beta\\
&\leq \sum_{b \in B} \sum_{\bf{h} \in \calE_n, \bf{b} \in \calB_n} |\mu(\{\tilde\xi_0=b\}\cap C_{\bf{h}}^{\bf{b}}) - \mu(\xi_0 = b) \mu(C_{\bf{h}}^{\bf{b}})| + 6 \beta.
\end{align*}
We recall that $f_n(\bf{b}, b')$ is the frequency at which the element $b'$ appears in the sequence $\bf{b}$ (see \eqref{eq:def_frequency}). Moreover, one can see that, since $\phi_{\bf{h}}(\bf{b})$ is the $\tilde\xi$-name of $C_{\bf{h}}^{\bf{b}}$ and all the levels of $C_{\bf{h}}^{\bf{b}}$ have the same measure, we have 
$$\mu(\{\tilde\xi_0=b\}\cap C_{\bf{h}}^{\bf{b}}) = \mu(C_{\bf{h}}^{\bf{b}}) \cdot f_n(\phi_{\bf{h}}(\bf{b}), b).$$
Therefore, because $\phi_{\bf{h}}$ takes values in $\calB_n$, \eqref{eq:law_ergo_thm} yields:
\begin{align*}
\sum_{b \in B}  |\mu(\tilde\xi_0 = b) - \mu(\xi_0 = b)| &\leq \sum_{b \in B} \sum_{\bf{h} \in \calE_n, \bf{b} \in \calB_n} \mu(C_{\bf{h}}^{\bf{b}}) |f_n(\phi_{\bf{h}}(\bf{b}), b) - \mu(\xi_0 = b)| + 6 \beta\\
&\leq \beta \mu(\calT_n) + 6\beta \leq 7\beta.
\end{align*}
This means that $\bar{d}_1(\calL(\tilde\xi_0), \calL(\xi_0)) \leq 7 \beta \leq \a$ (using \ref{d}).

We now turn our attention to the entropy of $\calH \vee \tilde\xi$. The $\tilde\xi$-name of a column $C_{\bf{h}}^{\bf{b}}$ is $\psi_{\bf{h}}(\bf{b}_{n_1})$, and since $\psi_{\bf{h}}$ is invective, we can deduce $\bf{b}_{n_1}$ from the $\tilde\xi$-name of $C_{\bf{h}}^{\bf{b}}$. This means that, on the levels of the truncated tower $\calT_{n_1} := \{T^jF\}_{0 \leq j \leq n_1-1}$, $\xi_0$ is $(\calH \vee \tilde\xi)_{[-n_1, n[}$-measurable. Indeed, if $x$ is in $\calT_{n_1}$ and the sequence $(\calH \vee \tilde\xi)_{[-n_1, n[}(x)$ is known, the sequence $\tilde\xi_{[-n_1, 0[}(x)$ must contain a ``$*$'', which indicates the moment the past orbit of $x$ passes trough $G$ before entering $\calT_n$. So the position of ``$*$'' in $\tilde\xi_{[-n_1, 0[}(x)$ tells us the index of the level of $\calT_{n_1}$ the point $x$ is on, which we call $j_0$. In other words, we mean that $T^{-j_0}x \in F$. Then, $(\calH \vee \tilde\xi)_{[-j_0, n - j_0[}$ gives the $(\calH \vee \tilde\xi)$-name of the column $x$ is on, from which we deduce the truncated $\xi$-name of length $n_1$ of the column. Finally, the $j_0$-th letter of that name gives us $\xi_0(x)$. 

Therefore, if we combine the previous paragraph with the fact that $\mu(\calT_{n_1}) \geq (1-\g)\mu(\calT_n)$, there exists a $(\calH \vee \tilde\xi)_{[-n_1, n[}$-measurable random variable $\chi_0$ such that $\mu(\chi_0 \neq \xi_0) \leq \beta + \gamma \leq 2\g$ (because $\beta \leq \g$). So, by the choice of $\g$ made in \ref{c}, we can apply Lemma \ref{lem:continuity_entropy} and conclude that
\begin{align*}
h_\mu(\calH \vee \xi, T) &\leq h_\mu(\calH \vee \chi, T) + \a \leq h_\mu(\calH \vee \tilde\xi, T) + \a.
\end{align*}
Because $\calH \vee \xi$ generates $\A$, we also have the converse inequality 
$$h_\mu(\calH \vee \tilde\xi, T) \leq h_\mu(\calH \vee \xi, T),$$
so we have proved that $\tilde \xi$ satisfies condition (ii) of our lemma. 

We are now left with proving (iii). If we consider that $\tilde\xi_{[-n, n[}(x)$ is known, we deduce that, if the symbol ``$*$'' appears in $\tilde\xi_{[-n, 0[}$, then $x$ is in $\calT_n$ and the position of ``$*$'' tells us the index $j_0$ of the level of $\calT_n$ the point $x$ is on. Then, using the notation introduced in our construction above, we can look at the random variable $\tilde\calP_{j_0}(\bf{h}_0, \tilde\xi_{[-j_0, n_1 - j_0[})$. It is $\tilde\xi_{[-n, n[}$-measurable and we are going to show that it satisfies
\begin{eq} \label{eq:goal}
\mu(\tilde\calP_{j_0}(\bf{h}_0, \tilde\xi_{[-j_0, n_1 - j_0[}) \neq \calP_0) \leq 5 \eps. 
\end{eq}
We start by looking at a column $C_{\bf{h}}$ for some $\bf{h} \in \bar{\calE}_n$. We then split it into sub-columns $C_{\bf{h}}^{\bf{b}}$. If $\bf{b}_{n_1} \in \bar{\calB}_{n_1}(\bf{h})$, we are going to use \eqref{eq:marriage_condition}. First, we need to remember that if $x$ is in $C_{\bf{h}}^{\bf{b}}$, then $\tilde\xi_{[-j_0, n-j_0[}(x)$ gives the $\tilde\xi$-name of the column $C_{\bf{h}}^{\bf{b}}$. But, by construction, that name is $\psi_{\bf{h}}(\bf{b}_{n_1})$, and, because we are looking at the case where $\bf{h} \in \bar{\calE}_n$ and $\bf{b}_{n_1} \in \bar{\calB}_{n_1}(\bf{h})$,  \eqref{eq:marriage_condition} holds. So we have
$$d_\ell(\tilde\calP_{L_n}(\bf{h}_0, \tilde\xi_{[-j_0, n_1 - j_0[}), \tilde\calP_{L_n}(\bf{h}, \bf{b}_{n_1})) \leq \eps.$$
We recall that $L_n = [n_0, n_1-n_0[$ and $\ell$ is its length. By definition of $d_\ell$, we know that the number of levels $j_0 \in L_n$ on which we have $\tilde\calP_{j_0}(\bf{h}_0, \tilde\xi_{[-j_0, n_1 - j_0[}) = \tilde\calP_{j_0}(\bf{h}, \bf{b}_{n_1})$ is greater than $(1- \eps)\ell$. Moreover, by construction, for $j_0 \in L_n$, on the $j_0$-th level of $C_{\bf{h}}^{\bf{b}}$, we have $\tilde\calP_0 = \tilde\calP_{j_0}(\bf{h}, \bf{b}_{n_1})$. Finally, since $\ell = n_1 - 2n_0$, we have
\begin{align*}
\mu(\tilde\calP_{j_0}(\bf{h}_0, \tilde\xi_{[-j_0, n_1 - j_0[}) \neq \tilde\calP_0 \, | \, C_{\bf{h}}^{\bf{b}}) &\leq \frac{n - (1-\eps)(n_1-2n_0)}{n}\\
&\leq \frac{n-(1-\eps)(1-\g)n + 2n_0}{n}\\
&\leq \eps + \g + \frac{2n_0}{n} \leq 2 \eps,
\end{align*}
since $\g \leq \eps/2$ and we can assume that $n$ is large enough so that $2n_0/n \leq \eps/2$. Moreover, the fact that $\bf{h} \in \bar{\calE}_n$ implies that $\mu(\xi_{[0, n_1[} \in \bar{\calB}_{n_1}(\bf{h}) \, | \, \Q = \bf{h}) \geq 1-\eps$, and, combining it with \eqref{eq:indep_base_tower}, we can see that
$$\mu\left(\left.\bigcup_{\bf{b}_{n_1}  \notin \bar{\calB}_{n_1}(\bf{h})} C_{\bf{h}}^{\bf{b}} \, \right| \, C_{\bf{h}}\right) \leq \eps.$$
Therefore
$$\mu(\tilde\calP_{j_0}(\bf{h}_0, \tilde\xi_{[-j_0, n_1 - j_0[}) \neq \tilde\calP_0 \, | \, C_{\bf{h}}) \leq 3\eps.$$
Next
\begin{align*}
\mu(\tilde\calP_{j_0}(\bf{h}_0, \tilde\xi_{[-j_0, n_1 - j_0[}) \neq \tilde\calP_0) &\leq 3 \eps + \mu\left(\bigcup_{\bf{h}  \notin \bar{\calE}_n} C_{\bf{h}}\right) + \mu(\calT_n^c)\\
&\leq 3 \eps + \eps + \mu(\calT_n^c) \; \text{ using \eqref{eq:indep_base_tower} and \eqref{eq:set_extremality}}\\
&\leq 4\eps + 3 \beta.
\end{align*}
Finally, $\tilde\calP_0$ was chosen so that $\mu(\tilde\calP_0 \neq \calP_0) \leq \beta^2 \leq \beta$, since $\beta \leq \eps/4$, we have proven \eqref{eq:goal}, and therefore, up to replacing $\eps$ by $\eps/5$, we have shown that
$$\calP_0 \preceq_\eps \s(\tilde\xi).$$
\end{proof}

\subsection{Application of the technical lemma}

We are now left with proving Proposition \ref{prop:unique_weak_pinsker_bernoulli} using Lemma \ref{lem:unique_weak_pinsker_bernoulli}. This is done using some abstract results from Thouvenot \cite[Proposition 2', Proposition 3]{article_thouvenot_RFD}. We start by rewriting those results with our notation. We give a slight simplification, adapted to our setup. 

First, \cite[Proposition 2']{article_thouvenot_RFD} tells us that a process close enough to an i.i.d. process independent from $\scrH$ in law and entropy can be turned into an i.i.d. process independent from $\scrH$. 
\begin{prop} \label{prop:thouvenot_2}
Let $\bfX := (X, \A, \mu, T)$ be an ergodic system of finite entropy. Let $\calH$ be a finite valued process defined on $\bfX$ and $\rho$ be a probability measure on a finite alphabet $B$. For every $\eps > 0$, there exist $\a > 0$ such that if a random variable $\tilde\xi_0: X \arr B$ satisfies
\begin{enumerate} [label = (\roman*)]
\item $\bar{d}_1(\rho, \calL(\tilde\xi_0)) \leq \a$,
\item and $0 \leq h_\mu(\calH, T) + H(\rho) - h_\mu(\calH \vee \tilde \xi, T) \leq \a$,
\end{enumerate}
then there exists a random variable $\xi'_0$ of law $\rho$ such that the process $\xi' := (\xi'_0 \circ T^n)_{n \in \bbZ}$ is i.i.d., independent from $\calH$ and we have
$$\mu(\tilde\xi_0 \neq \xi'_0) \leq \eps.$$
\end{prop}

Next, \cite[Proposition 3]{article_thouvenot_RFD} tells us that on a system that is relatively Bernoulli over a factor $\scrH$, any i.i.d. process independent from $\scrH$ with the right entropy can be turned into an independent complement of $\scrH$:
\begin{prop} \label{prop:thouvenot_3}
Let $\bfX := (X, \A, \mu, T)$ be an ergodic system, $\calH$ a finite valued process and $\xi$ a finite valued i.i.d. process independent from $\calH$ such that $\A = \s(\calH) \vee \s(\xi)$ mod $\mu$. For any $\eps > 0$ and any i.i.d. process $\zeta$ independent from $\calH$ such that $h_\mu(\xi, T) = h_\mu(\zeta, T)$, there exists $\tilde\zeta_0$ such that $\calL(\calH \vee \tilde \zeta) = \calL(\calH \vee \zeta)$, $\A = \s(\calH) \vee \s(\tilde\zeta)$ mod $\mu$ and 
$$\mu(\tilde \zeta_0 \neq \zeta_0) \leq \eps.$$
\end{prop}
We are now fully equipped to end the proof of Proposition \ref{prop:unique_weak_pinsker_bernoulli}:
\begin{proof}[Proof of Proposition \ref{prop:unique_weak_pinsker_bernoulli}]
Let $\bfX := (X, \A, \mu, T)$ be a Bernoulli shift of finite entropy and $\calP_0: X \arr A$ a finite valued random variable such that $\A = \s(\{\calP_0 \circ T^n\}_{n \in \bbZ})$. As we consider a factor $\s$-algebra $\scrH$ of $\bfX$, it has finite entropy, therefore there exists a finite valued random variable $\calH_0: X \arr H$ such that the process $\calH := (\calH_0 \circ T^n)_{n \in \bbZ}$ generates $\scrH$. Lastly, we take an i.i.d. process $\xi$ independent from $\scrH$ such that $\A = \scrH \vee \s(\xi)$ mod $\mu$. Let $\eps > 0$.

Now, Lemma \ref{lem:unique_weak_pinsker_bernoulli} tells us that there is $\delta > 0$ for which, if $h_\mu(\scrH, T) \leq \delta$, then for any $\a >0$, there is a random variable $\tilde\xi_0$ such that 
\begin{enumerate} [label = (\roman*)]
\item $\bar{d}_1(\calL(\xi_0), \calL(\tilde\xi_0)) \leq \a$,
\item $0 \leq h_\mu(\calH \vee \xi, T) - h_\mu(\calH \vee \tilde \xi, T) \leq \a$,
\item and $\calP_0 \preceq_{\eps/4} \s(\tilde\xi)$.
\end{enumerate}
Denote $\tilde\calP_0$ a $\tilde\xi$-measurable random variable such that $\mu(\tilde\calP_0 \neq \calP_0) \leq \eps/4$. We can find an integer $N \geq 1$ for which $\tilde\calP_0 \preceq_{\eps/4} \tilde\xi_{[-N, N]}$ and set $\eps_1 := \eps/(4(2N+1))$. If $\a$ is chosen small enough, then Proposition \ref{prop:thouvenot_2} tells us that there is a random variable $\xi'_0$ such that the process $(\xi'_0\circ T^n)_{n \in \bbZ}$ is i.i.d., independent from $\calH$ and we have $\mu(\xi'_0 \neq \tilde\xi_0) \leq \eps_1$. Finally, Proposition \ref{prop:thouvenot_3} tells us that we can then find a random variable $\xi''_0$ for which the process $(\xi''_0\circ T^n)_{n \in \bbZ}$ is still i.i.d., independent from $\calH$, but we also have that $\A = \scrH \vee \s(\xi'')$ mod $\mu$ and $\mu(\xi'_0 \neq \tilde\xi_0) \leq \eps_1$. So we have $\mu(\xi''_0 \neq \tilde\xi_0) \leq 2 \eps_1$. 

Combining that with the fact that $\tilde\calP_0 \preceq_{\eps/4} \tilde\xi_{[-N, N]}$, we get that $\tilde\calP_0 \preceq_{3\eps/4} \xi''_{[-N, N]}$, so
$$\calP_0 \preceq_{\eps} \xi''_{[-N, N]}.$$
Setting $\B := \s(\xi'')$, we get the Bernoulli factor desired to prove our proposition. 
\end{proof}

\subsection{Proof of Theorem \ref{thm:unique_weak_pinsker_bernoulli}}
\label{sect:proof_of_thm_331}

In the previous section, we managed to conclude the proof of Proposition \ref{prop:unique_weak_pinsker_bernoulli}. We now see how Theorem \ref{thm:unique_weak_pinsker_bernoulli} follows from that proposition:
\begin{proof}[Proof of Theorem \ref{thm:unique_weak_pinsker_bernoulli}]
Let $\bfX :=(X, \A, \mu,T)$ be a Bernoulli system and $\F := (\F_n)_{n \leq 0}$ be a weak Pinsker filtration. Since $\F$ is a weak Pinsker filtration, if $(\F_n)_{ \leq -1}$ is of product type, so is $\F$. Therefore, up to replacing $\bfX$ by the factor generated by $\F_{-1}$, we can assume that $\bfX$ has finite entropy. Thanks to Theorem \ref{thm:krieger}, this means that we can set a finite alphabet $A$ and a random variable $\calP_0: X \arr A$ such that the corresponding process $\calP := (\calP_0\circ T^i)_{i \in \bbZ}$ generates $\A$, i.e. $\A = \s(\calP)$ mod $\mu$. Let $(\eps_k)_{k \geq 1}$ be a decreasing sequence of positive numbers such that $\lim_{k\arr \infty} \eps_k = 0$. 

We need to build a strictly increasing sequence $(n_k)_{k \leq 0}$ such that $(\F_{n_k})_{k \leq 0}$ is of product type. We start by setting $n_0 = 0$. Since $\lim_{n \arr -\infty}h_\mu(\F_n, T) =0$, we can choose $n_{-1} \leq 0$ large enough (in absolute value), so that $h_\mu(\F_{n_{-1}}, T)$ is small enough for Proposition \ref{prop:unique_weak_pinsker_bernoulli} to enable us to build a Bernoulli factor $\s$-algebra $\B_{n_{-1}}$ that is an independent complement of $\F_{n_{-1}}$ such that $\calP_0 \preceq_{\eps_1} \B_{n_{-1}}$. 

Now take $k \leq -1$ and assume that we have built $(\B_{n_{-1}}, ..., \B_{n_k})$ such that they are mutually independent Bernoulli factors such that for $k \leq j  \leq -1$, $\B_{n_j}$ is independent from $\F_{n_j}$, $\F_{n_{j+1}} = \F_{n_j} \vee \B_{n_j}$ and we have
\begin{eq} \label{eq:main_approx_P}
\calP_0 \preceq_{\eps_{|k|}} \bigvee_{k \leq j \leq -1} \B_{n_j}.
\end{eq}
By construction of the $\B_{n_j}$, we know that $\calP$ is measurable with respect to $\F_{n_k} \vee \bigvee_{k \leq j \leq -1} \B_{n_j}$. Moreover, using again Theorem \ref{thm:krieger}, there is a random variable $\calP^{(k)}_0$ such that the process $\calP^{(k)} := (\calP^{(k)}_0 \circ T^i)_{i \in \bbZ}$ generates $\F_{n_k}$. So there exists an integer $N \geq 1$ such that 
\begin{eq} \label{eq:P_approx}
\calP_0 \preceq_{\eps_{|k|+1}/2} \calP^{(k)}_{[-N, N]} \vee \bigvee_{k \leq j \leq -1} \B_{n_j}.
\end{eq}
Then set $\tilde\eps := \eps_{|k|+1}/(2(2N+1)) > 0$. As we did above, we choose $n_{k-1} \leq n_k$ large enough in absolute value so that  $h_\mu(\F_{n_{k-1}}, T)$ is small enough for us to apply Proposition \ref{prop:unique_weak_pinsker_bernoulli} to find a Bernoulli factor $\B_{n_{k-1}} \subset \F_{n_k}$ such that $\B_{n_{k-1}} {\indep} \F_{n_{k-1}}$, $\F_{n_k} = \F_{n_{k-1}} \vee \B_{n_{k-1}}$ and 
\begin{eq} \label{eq:F_approx}
\calP^{(k)}_0 \preceq_{\tilde \eps} \B_{n_{k-1}}.
\end{eq}
Putting \eqref{eq:P_approx} and \eqref{eq:F_approx} together, we get
$$\calP_0 \preceq_{\eps_{|k|+1}} \bigvee_{k-1 \leq j \leq -1} \B_{n_j}.$$

Iterating this for every $k \leq -1$ ends our construction of $(n_k)_{k \leq 0}$ and $(\B_{n_k})_{k \leq -1}$. Therefore \eqref{eq:main_approx_P} holds for every $k \leq -1$. It follows then that $\calP_0$ is measurable with respect to 
$$\bigvee_{j \leq -1} \B_{n_j}.$$
Since the $\B_{n_j}$ are factor $\s$-algebras, the full process $\calP$ is also $\bigvee_{j \leq -1} \B_{n_j}$-measurable. Finally, $\calP$ generates $\A$, so 
$$\bigvee_{j \leq -1} \B_{n_j} = \A = \F_0 \, \text{ mod } \mu.$$

Let $k \leq -1$, and set $\E_1:= \bigvee_{j \leq k-1} \B_{n_j}$ and $\E_2 := \bigvee_{k \leq j \leq -1} \B_{n_j}$. By construction, we have 
$$\E_1 \subset \F_{n_k}, \; \F_{n_k} \indep \E_2, \; \text{ and } \; \F_0 = \E_1 \vee \E_2.$$
We use this to see that if $f$ is $\F_{n_k}$-measurable, we have
$$f = \bbE[f \, | \, \F_0] = \bbE[f \, | \, \E_1 \vee \E_2] = \bbE[f \, | \, \E_1],$$
which proves that 
$$\F_{n_k} = \E_1 = \bigvee_{j \leq k-1} \B_{n_j} \; \text{ mod } \mu.$$
\end{proof}

\section{Examples of weak Pinsker filtrations generated by a cellular automaton}
\label{sect:example_weak_pinsker}

Up to this point, we have discussed the existence and abstract properties of weak Pinsker filtrations. Now we want to give explicit examples to get a more concrete idea of what those objects can look like. We take inspiration from \cite{article_PL} and use cellular automata to generate our filtrations. We describe in the following paragraphs how this is done. 

Let $A$ be a finite alphabet. A cellular automaton (or, more precisely, a deterministic cellular automaton) $\tau: A^\bbZ \arr A^\bbZ$ maps $A^\bbZ$ onto itself as follows: take $F \subset \bbZ$ finite, which we call a neighborhood, and a local map $\tau_0: A^F \arr A$. Then define 
$$\tau: (a_n)_{n \in \bbZ} \mapsto (\tau_0((a_{n+k})_{k \in F}))_{n \in \bbZ}.$$
Here, we will only consider examples in which $F = \{0, 1\}$. Therefore, our automata will be determined by a local map of the form $\tau_0: A^{\{0, 1\}} \arr A$. One can note that, by construction, cellular automata commute with the shift transformation
$$S: (a_n)_{n \in \bbZ} \mapsto (a_{n+1})_{n \in \bbZ}.$$
So we can consider a dynamical system of the form $\bfY := (A^\bbZ, \B, \nu, S)$ where $\nu$ is a $S$-invariant measure, and note that the $\s$-algebra $\s(\tau)$ generated by $\tau$ is a factor $\s$-algebra. We can do better and iterate $\tau$ to generate a filtration: 
$$\text{for } n \leq 0, \; \F_n := \s(\tau^{|n|}).$$ 
In that case, each $\F_n$ is a factor $\s$-algebra of $\bfY$, and therefore $\F := (\F_n)_{n \leq 0}$ is a \emph{dynamical filtration}. So, we see that cellular automata give a natural way to construct dynamical filtrations. 

In fact, the theory of dynamical filtrations we presented in Section \ref{sect:dyn_filtr} was initiated in \cite{article_PL} in the setting of filtrations generated by cellular automata. However, the automata studied there preserve the product measure, and therefore the entropy of the associated factor $\s$-algebras $\F_n$ will be the same for every $n \leq 0$. This prevents the filtration from being weak Pinsker. 

Here, we will consider a different automaton: take $A$ a finite alphabet and assume that one element of $A$ is labeled <<$0$>>. Then define the following local map 
\begin{eq}\label{eq:local_function}
\begin{array}{cccc}
	\tau_0: & A^2 &\to & A \\
	& (\alpha_1, \alpha_2) &\mapsto &  \left\{ 
\begin{array}{ll}
	\alpha_1 & \text{ if } \, \alpha_1 = \alpha_2\\
	0 & \text{ otherwise}.
\end{array}
\right.
\end{array}
\end{eq}
The associated automaton will eliminate isolated elements, replacing them with $0$, and a maximal string of the form $\a \cdots \a \a$ is replaced with $\a \cdots \a 0$. For example, if $A = \{0, 1\}$, this gives:

\centerline{\includegraphics[scale=1]{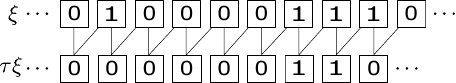}}

Therefore, as we iterate the automaton, the proportion of <<$0$>> increases as all other elements are gradually replaced by <<$0$>>. Heuristically, this indicates that the entropy of the factor $\s$-algebras $\s(\tau^{|n|})$ will go to zero as $n$ goes to infinity. But to state this rigorously, one need to specify the system $\bfY := (A^\bbZ, \B, \nu, S)$ on which we define $\F$. More accurately, it is the alphabet $A$ and the measure $\nu$ that need to be specified. However, the entropy $h_\nu(\F_n)$ goes to $0$ regardless of the choice of $A$ and $\nu$:
\begin{prop} \label{prop:general_entropy}
Let $\bfY := (A^{\bbZ}, \B, \nu, S)$, where $\nu$ is a $S$-invariant measure and let $\xi$ be the coordinate process on $\bfY$. For every $n \geq 1$, we have 
$$h_\nu(\tau^n\xi, S) \leq \frac{\log(\#A n^2)}{n}.$$
\end{prop}
\begin{proof}
Let $B_n \subset A^n$ be the set of values taken by $(\tau^n\xi)_{[0, n[}$. We know that 
$$H_\nu((\tau^n\xi)_{[0, n[}) \leq \log(\#B_n).$$
Because of the structure of $\tau$, in $\tau^n\xi$, for $\a \neq 0$, any run of <<$\a$>> is placed in between two runs of <<$0$>> of length at least $n+1$. Therefore, $(\tau^n\xi)_{[0, n[}$ is either a sequence of <<$0$>> or composed of one run of <<$\a$>> (with $\a \neq 0$) in between runs of <<$0$>>. So
$$\#B_n \leq 1 + (\#A -1)n^2 \leq \#A n^2.$$
In conclusion
$$h_\nu(\tau^n\xi, S) \leq \frac{1}{n} H_\nu((\tau^n\xi)_{[0, n[}) \leq \frac{\log(\#A n^2)}{n}.$$
\end{proof}

In Section \ref{sect:cellular_b_shift}, we deal with the case where $\bfY$ is a Bernoulli shift, and in Section \ref{sect:ornstein_k_non-b}, we deal with the case where $\bfY$ is Ornstein's example of a non-Bernoulli K-system from \cite{exemple_K-sys_non_bernoulli}. In both cases, by Proposition \ref{prop:general_entropy}, the entropy of the filtration generated by the cellular automaton goes to zero. Then we look at each example separately to show the more involved result: each $\F_{n+1}$ is relatively Bernoulli over $\F_n$. Therefore, we get two examples of weak Pinsker filtrations. 

It is interesting to note that those two filtrations are very similar in their construction, but the filtration (or any sub-sequence) on Ornstein's K-system cannot be of product type (otherwise, the system would be Bernoulli), we know from Theorem \ref{thm:unique_weak_pinsker_bernoulli} that the latter has a sub-sequence that is of product type. It shows that there can be subtle differences in the asymptotic structure of weak Pinsker filtrations.
 
\subsection{A cellular automaton on a Bernoulli shift}
\label{sect:cellular_b_shift}

Here, we consider a Bernoulli shift $\bfY := (A^\bbZ, \B, \nu, S)$ where $\nu$ is a product measure. To avoid unnecessarily complicated notations, we will also assume that $A = \{0, 1\}$ and $\nu := (\frac{1}{2}(\delta_0 + \delta_1))^{\otimes \bbZ}$. Therefore, the local function \eqref{eq:local_function} becomes:
$$
\begin{array}{cccc}
	\tau_0: & \{0, 1\}^2 &\to & \{0, 1\} \\
	& \alpha &\mapsto &  \left\{ 
\begin{array}{ll}
	1 & \text{ if } \, \alpha = (1, 1)\\
	0 & \text{ otherwise}.
\end{array}
\right.
\end{array}
$$
And we study the corresponding automaton:
$$\begin{array}{cccc}
	\tau: & \{0, 1\}^\bbZ & \to & \{0, 1\}^\bbZ\\
	& (a_n)_{n \in \bbZ} & \mapsto & (\tau_0(a_n, a_{n+1}))_{n \in \bbZ}\\
\end{array}$$
\noindent The automaton replaces an isolated <<$1$>> with a <<$0$>> and reduces sequences of <<$1$>>  by replacing the final one by a <<$0$>>.

\begin{thm} \label{thm:bernoulli_weak_pinsker}
On the system $\bfY := (\{0, 1\}^{\bbZ}, \B, \nu, S)$, the filtration given by $\F := (\s(\tau^{|n|}))_{n \leq 0}$ is a weak Pinsker filtration. That is, for every $n \leq -1$, $\F_{n+1}$ is relatively Bernoulli over $\F_{n}$ and we have
\begin{eq} \label{eq:limite_entropie}
h_\nu(\F_n) \underset{n \rightarrow - \infty}{\to} 0.
\end{eq}
\end{thm}

The convergence of the entropy follows from Proposition \ref{prop:general_entropy}. However, when $\bfY$ is a Bernoulli shift, we can compute a better bound, as stated in Proposition \ref{prop:majoration_entropie}.

\begin{prop} \label{prop:majoration_entropie}
Let $\xi$ denote the coordinate process on $\bfY$. For every $n \geq 0$, we have
$$h_\nu(\tau^n\xi, S) \leq 3 \log(2) 2^{-n/2}.$$
\end{prop}
\begin{proof}
Let $n \geq 0$. One can see that $\tau^n\xi$ is $1$ at $i$ if and only if $\xi$ is $1$ over the entire segment $[i, i+n]$, as shown below:

\centerline{\includegraphics[scale=0.17]{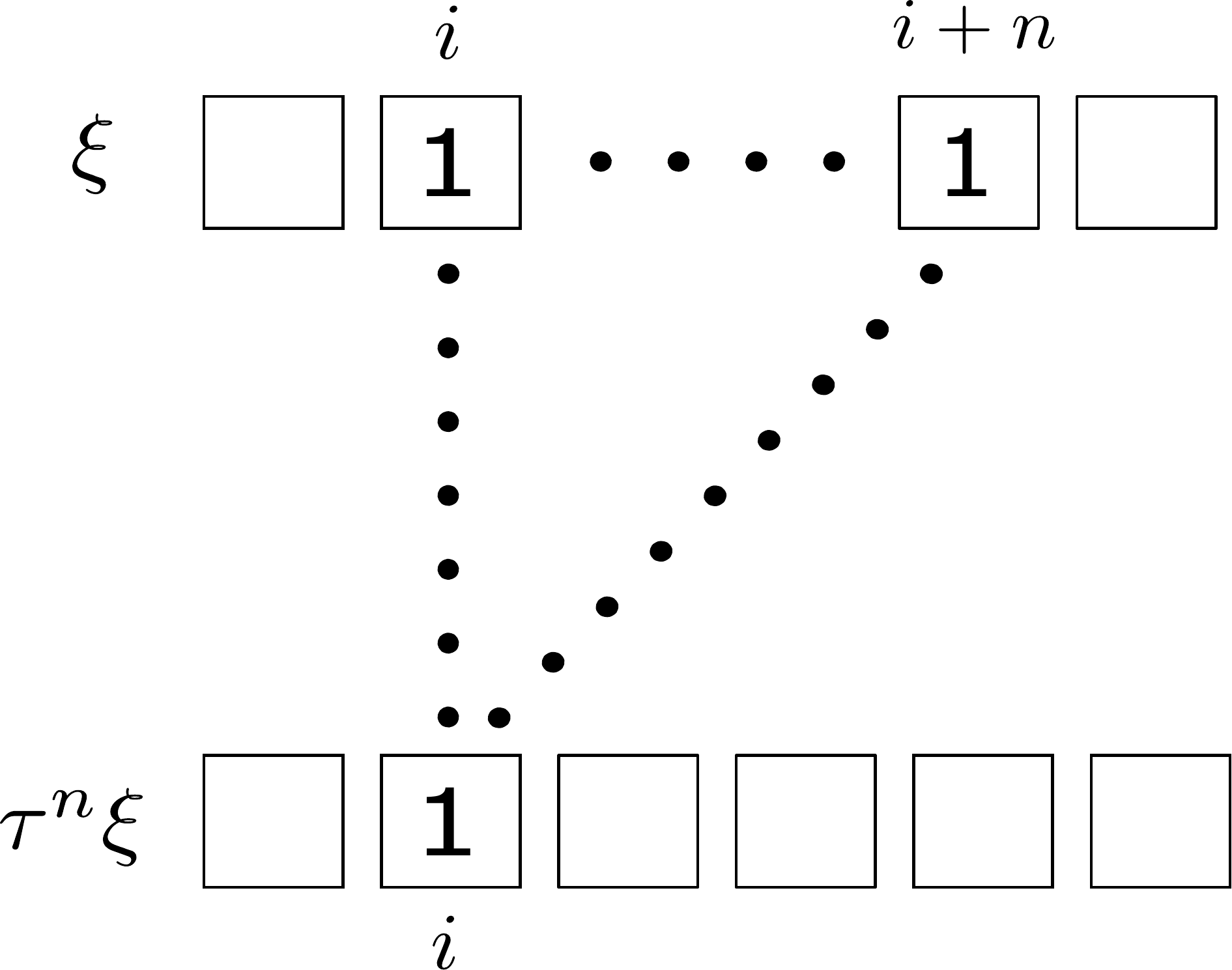}}

We set $k_n := \lceil n/2 \rceil$, and we remark that
$$\nu(\{\exists i \in [0, k_n], (\tau^n\xi)_i = 1\}) \leq \nu(\{\xi_{[k_n, n]} = (1, ..., 1)\}) \leq 1/2^{n-k_n+1} \leq 1/2^{n/2}.$$

Then, combining this with Fano's inequality (Lemma \ref{lem:Fano}) we get
$$H_\nu((\tau^n\xi)_{[0, k_n]}) \leq 2^{-n/2} (1 + \log(2^{k_n+1}) + \log(2^{n/2})) \leq 2^{-n/2} 3 (k_n+1) \log(2),$$
and we can conclude for the KS-entropy:
$$h_\nu(\tau^n\xi, S) \leq \frac{1}{k_n+1} H_\nu((\tau^n\xi)_{[0, k_n]}) \leq 3 \log(2) 2^{-n/2}.$$
\end{proof}

In addition, we give the following simple lemma on conditional independence:
\begin{lem} \label{lem:conditional_indep}
Let $(X, \A, \mu)$ be a probability space and $\scrZ$ a sub-$\s$-algebra. Let $A$, $B$, $U$ and $V$ be random variables such that 
$$(A, U) \indep_\scrZ (B, V).$$
Then we have
\begin{align*}
\calL(A, B \, | \, U, V, \scrZ) &= \calL(A \, | \, U, \scrZ) \otimes \calL(B \, | \, V, \scrZ)\\
&= \calL(A \, | \, U, V, \scrZ) \otimes \calL(B \, | \, U, V, \scrZ)
\end{align*}
\end{lem}
\begin{proof}
It follows from the fact that if $A'$, $B'$, $U'$ and $V'$ are respectively $A$, $B$, $U$ and $V$-measurable random variables:
\begin{align*}
\bbE[A' \cdot B' \cdot U' \cdot V' \, | \, \scrZ] &= \bbE[A' \cdot U' \, | \, \scrZ] \cdot \bbE[ B' \cdot V' \, | \, \scrZ]\\
&= \bbE[\bbE[A' \, | \, U, \scrZ] \cdot U' \, | \, \scrZ] \cdot \bbE[\bbE[B' \, | \, V, \scrZ] \cdot V' \, | \, \scrZ]\\
&= \bbE[\bbE[A' \, | \, U, \scrZ] \cdot U' \cdot \bbE[B' \, | \, V, \scrZ] \cdot V' \, | \, \scrZ]. 
\end{align*}
\end{proof}

\begin{prop} \label{prop:VWB_rel}
Let $\xi$ be the coordinate process on $\bfY$. For every $n \geq 1$, $\xi$ is relatively very weak Bernoulli over $\tau^n\xi$.
\end{prop}

\begin{proof}
Set $\eta := \tau^n\xi$. Relative very weak Bernoullicity was defined in Definition \ref{def:VWB_rel}. We recall some notation: take $\lambda \in \P(\{0, 1\}^\bbZ \times \{0, 1\}^\bbZ)$ to be the law of $(\eta, \xi)$, and for $I, J \subset \bbZ$ and $a, b \in \{0, 1\}^\bbZ$, $\lambda_\ell(\cdot \, | \, a_I, b_J)$ is the conditional law of $\xi_{[0, \ell[}$ given that $\eta_I = a_I$ and $\xi_J = b_J$.

Let $\eps > 0$. We need to show that there exists $\ell \geq 1$ such that for every $m \geq 1$ and for $k \geq 1$ large enough, we have
\begin{eq} \label{eq:rel_vwb_proof}
\int \bar{d_\ell}\left(\lambda_\ell(\cdot \, | \, a_{[-k, k]}, b_{[-m, 0]}), \lambda_\ell(\cdot \, | \, a_{[-k, k]})\right) d\lambda(a, b) \leq \varepsilon.
\end{eq}

Let $m \geq 1$. We start by noting that there must be some <<$1$>> that appears in $\eta$: indeed, the law of large numbers tells us that there exists $\ell_0 \geq 1$ such that
\begin{eq} \label{eq:existence_blocs}
\mu(\underset{:= A}{\underbrace{\{\exists i \in [0, \ell_0[ \, ; \; \eta_i =1\}}}) \geq 1 - \varepsilon.
\end{eq}
We then set $\ell := \lceil \frac{1}{\varepsilon} \rceil \ell_0$. Next, we take $k \geq \ell_0$ so that $\eta_{[-k, k]}$ determines entirely $A$. 

We fix $i \in [0, \ell_0[$. First, we note that, as we can see on the following image

\centerline{\includegraphics[scale=0.06]{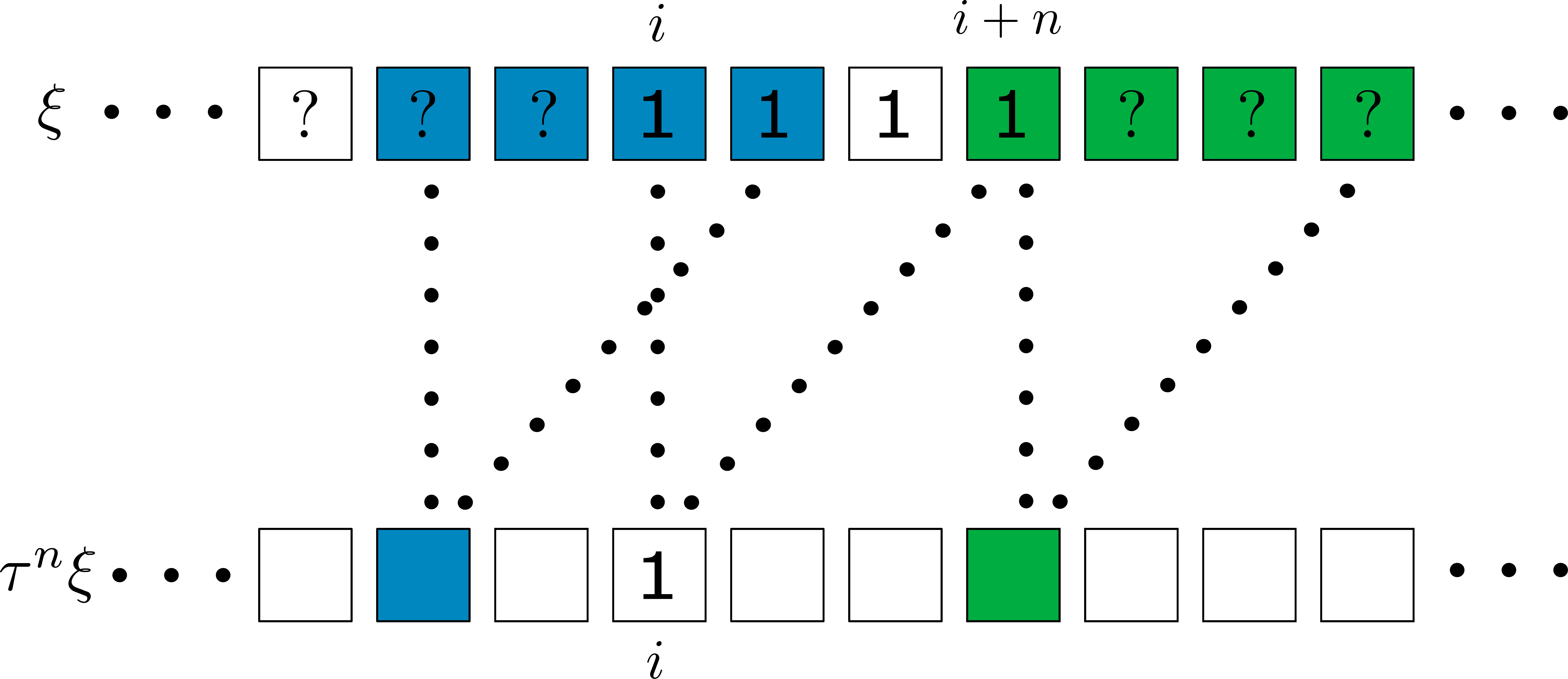}}
\noindent if $\eta_i = 1$, then $(\xi_{]-\infty, i[}, \eta_{]-\infty, i[})$ is $\xi_{]-\infty, i[}$-measurable and $(\xi_{]i, \infty[}, \eta_{]i, \infty[})$ is $\xi_{]i+n, \infty[}$-measurable. So, since the variables $\{\xi_j\}_{j \in \bbZ}$ are independent, given $\{\eta_i=1\}$ the variables $(\xi_{]-\infty, i[}, \eta_{]-\infty, i[})$ and $(\xi_{]i, \infty[}, \eta_{]i, \infty[})$ are independent. Finally, using Lemma \ref{lem:conditional_indep}, for $a \in A^{\bbZ}$ such that $a_i = 1$, we get:
\begin{align*}
\calL(\xi_{]-\infty, i[}, \xi_{]i, \infty[} \, | \, &\eta_{[-k, k]} = a_{[-k, k]})  = \calL(\xi_{]-\infty, i[} \, | \, \eta_{[-k, i[} = a_{[-k, i[}, \eta_i=1)\\
& \hspace{3cm} \otimes \calL(\xi_{]i, \infty[} \, | \, \eta_{]i, k]} = a_{]i, k]}, \eta_i = 1)\\
& = \calL(\xi_{]-\infty, i[}  \, | \, \eta_{[-k, k]} = a_{[-k, k]}) \otimes \calL(\xi_{]i, \infty[} \, | \, \eta_{]i, k]} = a_{[-k, k]}). 
\end{align*}
Therefore, if $\eta_{[-k, k]}$ is chosen so that there exists $i \in [0, \ell_0[$ such that $\eta_i=1$, we see that $\xi_{[-m, 0[}$ and $\xi_{[\ell_0, \ell[}$ are independent given $\eta_{[-k, k]}$.

We are now ready to prove \eqref{eq:rel_vwb_proof}. For any $b \in \{0, 1\}^\bbZ$ and any $a \in \{0, 1\}^\bbZ$ such that there exists $i \in [0, \ell_0[$ such that $a_{i} = 1$, the fact that $\xi_{[-m, 0]}$ and $\xi_{[\ell_0, \ell[}$ are relatively independent given $\{\eta_{[-k, k]} = a_{[-k, k]}\}$ implies that the measures $\lambda_\ell(\cdot \, | \, a_{[-k, k]}, b_{[-m, 0]})$ and $\lambda_\ell(\cdot \, | \, a_{[-k, k]})$ have the same marginal on the coordinates of $[\ell_0, \ell[$. So the relative product of those measures over $\xi_{[\ell_0, \ell[}$ is a coupling under which the copies of $\xi_{[\ell_0, \ell[}$ coincide. It follows that
\begin{eq} \label{eq:majoration_dbar}
\bar{d_\ell}\left(\lambda_\ell(\cdot \, | \, a_{[-k, k]}, b_{[-m, 0]}), \lambda_\ell(\cdot \, | \, a_{[-k, k]})\right) \leq \ell_0/\ell \leq \varepsilon.
\end{eq}

By combining \eqref{eq:existence_blocs} and \eqref{eq:majoration_dbar}, we can conclude that
$$\int \bar{d_\ell}\left(\lambda_\ell(\cdot \, | \, a_{[-k, k]}, b_{[-m, 0]}), \lambda_\ell(\cdot \, | \, a_{[-k, k]})\right) d\lambda(a, b) \leq 2 \varepsilon.$$
\end{proof}

\begin{proof}[Proof of Theorem \ref{thm:bernoulli_weak_pinsker}]
First of all, \eqref{eq:limite_entropie} follows directly from Proposition \ref{prop:majoration_entropie}.
Next, from Proposition \ref{prop:VWB_rel}, it follows that $\F_0$ is relatively very weak Bernoulli over $\F_n$, so $\F_{n+1}$ is relatively very weak Bernoulli over $\F_n$ (by part (iii) of Lemma \ref{lem:factors_rel_B_are_B}), so $\F_{n+1}$ is relatively Bernoulli over $\F_n$ (by part (i) of Lemma \ref{lem:factors_rel_B_are_B}).
\end{proof}

\subsection{A cellular automaton on Ornstein's K-process}
\label{sect:ornstein_k_non-b}

Here, we consider the non-Bernoulli K-system introduced by Ornstein in \cite{exemple_K-sys_non_bernoulli}. A more detailed presentation of this system is given in \cite[Part III]{book_ornstein}, but we give a sketch of the construction for completeness. It is a process defined on the alphabet $\{0, e, f, s\}$. We set $h(r)$, $s(r)$ and $f(r)$ to be integers depending on $r \in \bbN$ used in the construction of the process. For $r \geq 1$, an $r$-block is a random sequence of length $h(r)$ on the alphabet $\{0, e, f, s\}$, whose law we define inductively. 

To get a $1$-block, take $k_1 \in \llbracket 1, f(1)-1 \rrbracket$ chosen uniformly at random, and consider a sequence that starts with a string of $k_1$ <<$f$>>, followed by a string of $h(0)$ <<$0$>>, and ends with a string of $f(1) - k_1$ <<$e$>>: 
\vspace{3mm}

\centerline{\includegraphics[scale=0.4]{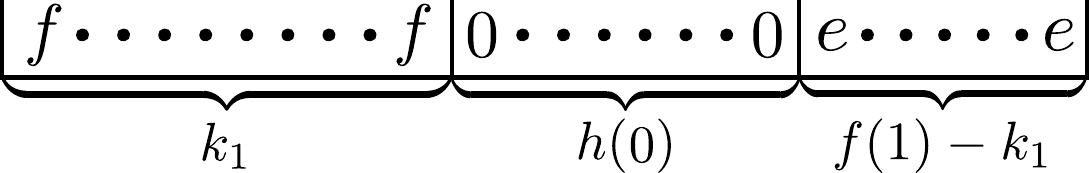}}
\noindent This construction implies that $h(1)=h(0)+ f(1)$.

To get an $r$-block, take $k_r \in \llbracket 1, f(r)-1 \rrbracket$ chosen uniformly at random, and $2^{r}$ i.i.d. random variables $(\xi^{(r-1)}_i)_{i \in \llbracket 1, 2^{r} \rrbracket}$ such that each $\xi^{(r-1)}_i$ is an $(r-1)$-block. The $r$-block is then built as follows:
\vspace{3mm}

\centerline{\includegraphics[scale=0.37]{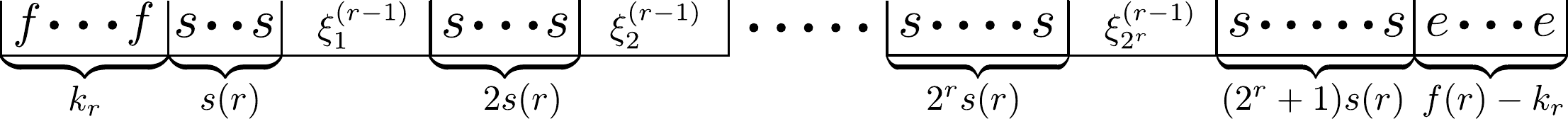}}
So an $r$-block starts with a string of $k_r$ <<$f$>>, and ends with a string of $f(r) - k_r$ <<$e$>>. In between, we put all the $(r-1)$-blocks separated by strings of <<$s$>> so that each $\xi_i^{(r-1)}$ is placed in between two strings of <<$s$>> of respective lengths $i s(r)$ and $(i+1) s(r)$. In particular, $h(r)$ is entirely determined by $h(r-1)$, $f(r)$ and $s(r)$. 

Ornstein's K-system is then built by constructing an increasing sequence of towers $(\calT_r)_{r \geq 1}$ such that $X := \bigcup_{r \geq 1} \calT_r$. A tower $\calT_r$ is given by its base $F_r$ for which the sets $\{T^iF_r\}_{i \in [0, h(r)[}$ are disjoint and 
$$\calT_r := \bigsqcup_{i=0}^{h(r) -1 } T^iF_r.$$
Through a cutting and stacking method, Ornstein builds in \cite{exemple_K-sys_non_bernoulli} the towers $(\calT_r)_{r \geq 1}$ along with a process $\xi$ so that the law of $\xi_{[0, h(r)[}$ given $F_r$ is the law of an $r$-block. In other words, this means that the columns of the form
$$C_\bfa := \bigsqcup_{i = 0}^{h(r)-1} T^i(F_r \cap \{ \xi_{[0, h(r)[} = \bfa\}), \, \text{ for } \bfa \in \{0, e, f, s\}^{h(r)},$$
partition $\calT_r$ according to the law of an $r$-block. Denote $\bfX := (X, \A, \mu, T)$ the resulting dynamical system. A proper choice of $h(r)$, $s(r)$ and $f(r)$ assures that this construction gives a finite measure. Then $\xi$ is a factor map onto the system 
$$\bfY := (\{0, e, f, s\}^\bbZ, \B, \nu, S),$$
where $\nu$ is the law of $\xi$.

Since $\xi$ is a process on the alphabet $\{0, e, f, s\}$, the local function \eqref{eq:local_function} becomes:
$$
\begin{array}{cccc}
	\tau_0: & \{0, e, f, s\}^2 &\to & \{0, e, f, s\} \\
	& (\alpha_1, \alpha_2) &\mapsto &  \left\{ 
\begin{array}{ll}
	\alpha_1 & \text{ if } \, \alpha_1 = \alpha_2\\
	0 & \text{ otherwise}.
\end{array}
\right.
\end{array}
$$
From now on, $\tau$ denotes the corresponding cellular automaton. Similarly to what we did in Section \ref{sect:cellular_b_shift}, we prove 
\begin{thm} \label{thm:ornstein_K_weak_pinsker}
On the system $\bfY := (\{0, e, f, s\}^\bbZ, \B, \nu, S)$, the filtration given by $\F := (\s(\tau^{|n|}))_{n \leq 0}$ is a weak Pinsker filtration. That is, for every $n \leq -1$, $\F_{n+1}$ is relatively Bernoulli over $\F_{n}$ and we have
\begin{eq} \label{eq:limite_entropie_K_sys}
h_\nu(\F_n) \underset{n \rightarrow - \infty}{\to} 0.
\end{eq}
\end{thm}
The overall structure of the proof will resemble Section \ref{sect:cellular_b_shift}, but the details are adapted to the specific structure of Ornstein's process. First, the convergence to $0$ of the entropy follows from Proposition \ref{prop:general_entropy}. We could also adapt the proof of Proposition \ref{prop:majoration_entropie} to get that convergence, but it does not give a better rate of convergence than Proposition \ref{prop:general_entropy}, so we do not give any details.

\begin{prop}
If $\xi$ is the process defined above, then for every $n \geq 1$, $\xi$ is relatively very weak Bernoulli over $\tau^n\xi$. 
\end{prop}

\begin{proof}
We set $\eta := \tau^n\xi$. Let $\varepsilon > 0$. Once again, we need to show that there exists $\ell \geq 1$ such that for every $m \geq 1$ and for $k \geq 1$ large enough, we have
$$\int \bar{d_\ell}\left(\lambda_\ell(\cdot \, | \, a_{[-k, k]}, b_{[-m, 0]}), \lambda_\ell(\cdot \, | \, a_{[-k, k]})\right) d\l(a, b) \leq \varepsilon,$$
where $\l$ is the law of $(\eta, \xi)$ and, for $I, J \subset \bbZ$, $\lambda_\ell(\cdot\, | \, a_I, b_J)$ is the conditional law of $\xi_{[0, \ell[}$ given that $\eta_I$ equals $a_I$ and that $\xi_J$ equals $b_J$.

Let $m \geq 1$. We choose $r$ so that $s(r+1) \geq n+1$. By construction of $\xi$, we know that for any $r$-block in $\xi$, there exists $i \in [1, 2^{r+1}]$ such that the said $r$-block will come after a string of $i \cdot s(r+1)$ <<$s$>> and be followed by a string of $(i+1)\cdot s(r+1)$ <<$s$>>. Therefore, by knowing the positions of all the strings of <<$s$>> longer that $s(r+1)$, we know the position of every $r$-block.

However, since we chose to have $s(r+1) \geq n+1$, we can say that, for $k \in \bbZ$, we have $\xi_{[k, k+s(r+1)[} = (s, ..., s)$ if and only if $\eta_{[k, k+s(r+1)-n[}=(s, ..., s)$. This means that the positions of the $r$-blocks contained on a segment $[k_1, k_2]$ are $\eta_{[k_1 - N, k_2 + N]}$-measurable, for $N$ large enough (for example $N = (2^{r+1}+1)s(r+1)$).

By choosing $r$ large enough, we can also assume that $\mu(\calT_r) \geq 1-\varepsilon/2$. Using Birkhoff's ergodic theorem, for $\ell$ large enough, the set 
$$A:= \left\{x \in X ; \, \frac{1}{\ell} \sum_{j=0}^{\ell-1}\mathbbm{1}_{\calT_r}(T^j(x)) > 1 - \eps \right\},$$
satisfies $\mu(A) > 1-\eps$. 

In other words, for $x \in A$, the number of elements in the sequence $\xi_{[0, \ell[}(x)$ that are part of an $r$-block is greater than $(1 - \eps)\ell$. However, among the intervals on which those $r$-blocks are supported, two of them may not be included in $[0, \ell[$, and can intersect $\bbZ \backslash [0, \ell[$.  But, if $h(r)/\ell \leq \eps/2$, there are at most $\eps \ell$ elements in those two intervals. To sum up, we get that the number of elements in the sequence $\xi_{[0, \ell[}(x)$ that are part of an $r$-block, and for which the position of that $r$-block is contained on the segment $[0, \ell[$, is greater than $(1 - 2\eps)\ell$. Then, we choose $k \geq 1$ so that the positions of the $r$-blocks contained in $[-m, \ell[$ are $\eta_{[-k, k]}$-measurable (in particular, $A$ is $\eta_{[-k, k]}$-measurable). So we have the following configuration for $\xi_{[-m, \ell[}$:

\centerline{\includegraphics[scale=0.3]{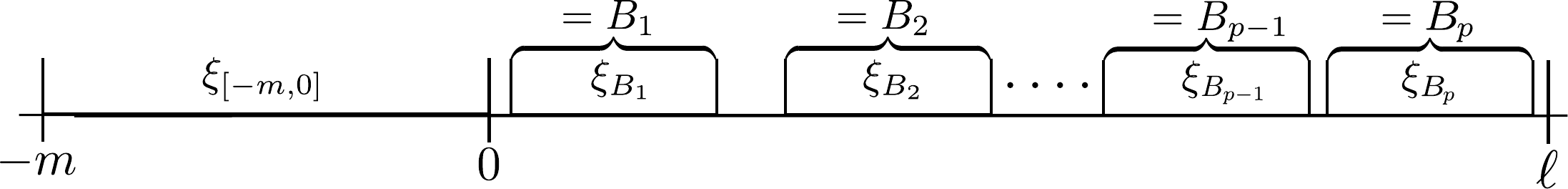}}
\noindent where the $B_i$ are the positions of the $r$-blocks supported on $[0, \ell[$, and we have shown that $\# \bigsqcup_{i = 1}^p B_i \geq (1-2\eps)\ell.$

Denote by $I_\ell := \{B_i\}_{1 \leq i \leq p}$ the random variable that gives the positions of the $r$-blocks on the segment $[0, \ell[$. By construction of $\xi$, we know that, given $I_\ell$, for any $r$-block $B$, the variables $\xi_B$ and $\xi_{B^c}$ are independent. Moreover, we know that any $r$-block is between two strings of at least $n+1$ <<$s$>>. Therefore, we see that if $I_\ell$ is fixed, for any $r$-block $B$, $\eta_B$ is $\xi_B$-measurable and $\eta_{B^c}$ is $\xi_{B^c}$-measurable.

Let us give details on the proof of that last claim: we write $B^c$ as the union of $B^-$ and $B^+$, the infinite intervals that come before and after $B$ respectively. Given the structure of our automaton, it is always true that $\eta_{B^+}$ is $\xi_{B^+}$-measurable. At the boundary between $B^-$ and $B$, we have the following configuration:

\centerline{\includegraphics[scale=0.06]{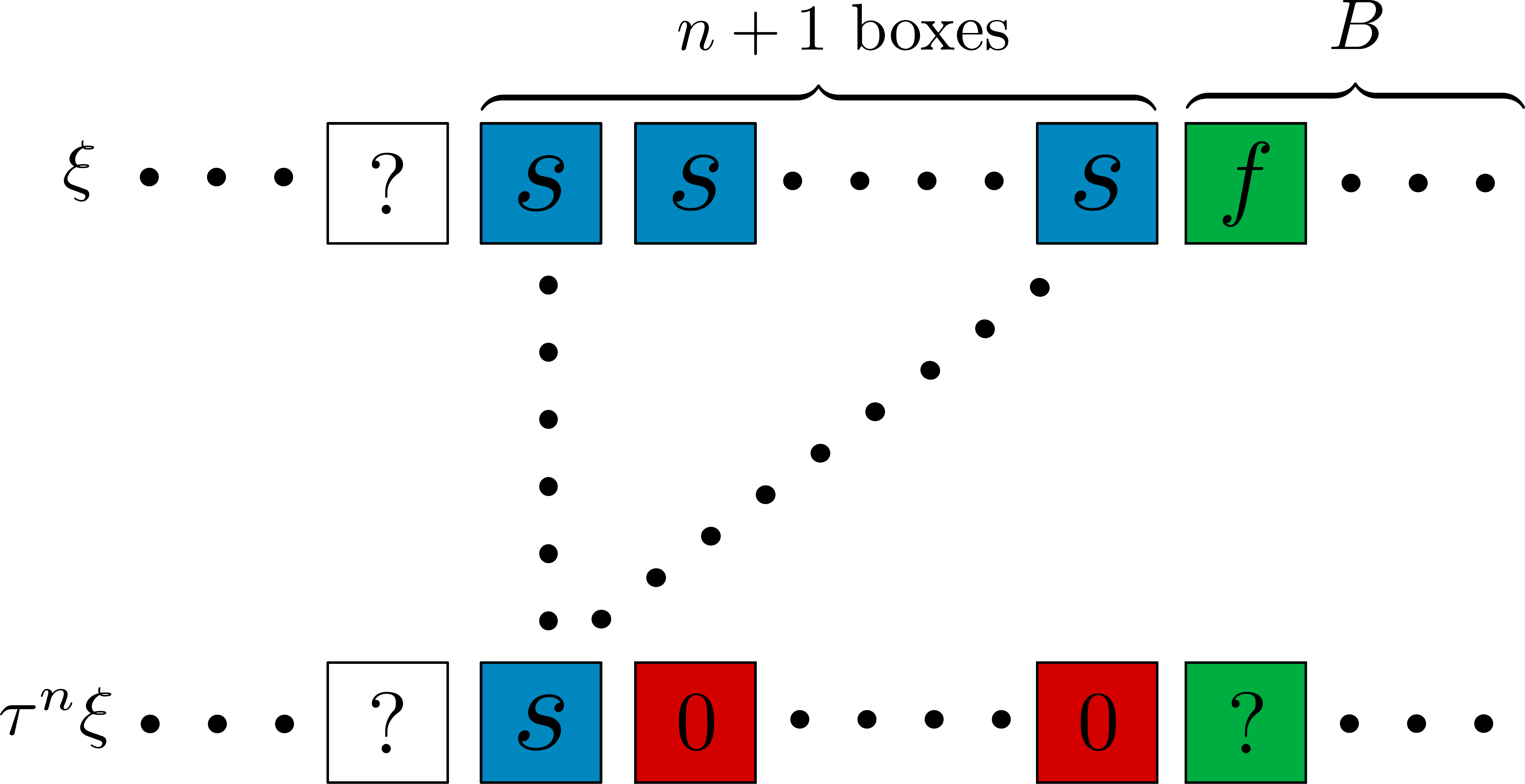}}
\noindent Indeed, in the construction of the blocks, we see that $\xi$ must put an <<$f$>> in the first box of $B$. Therefore, we must have <<$0$>> in the red boxes. So, the values that $\eta$ takes on the $n+1$ boxes preceding $B$ are determined. For the rest of the boxes of $B^-$, it comes from the structure of $\tau$ that the values of $\eta$ are determined by $\xi_{B^-}$ since we are at a distance $\geq n+1$ from $B$. So we have shown that $\eta_{B^-}$ is $(\xi_{B^-})$-measurable. A similar reasoning at the boundary between $B$ and $B^+$ shows that $\eta_B$ is $\xi_B$-measurable. And since it is always true that $\eta_{B^+}$ is $\xi_{B^+}$-measurable, we have proven that $\eta_B$ is $\xi_B$-measurable and $\eta_{B^c}$ is $\xi_{B^c}$-measurable.

But, we also know from the structure of $\xi$ that, given $I_\ell$, $\xi_B$ and $\xi_{B^c}$ are independent. The previous paragraph enables us to use Lemma \ref{lem:conditional_indep} to extend that to: given $I_\ell \vee \eta_{[-k, k]}$, $\xi_B$ and $\xi_{B^c}$ are independent. Finally, since $I_\ell$ is $\eta_{[-k, k ]}$-measurable, this yields that $\xi_{B}$ and $\xi_{B^c}$ are relatively independent given $\eta_{[-k, k]}$.

This independence tells us that, for every sequences $a$ and $b$, $\lambda_\ell(\cdot \, | \, a_{[-k, k]}, b_{[-m, 0]})$ and $\lambda_\ell(\cdot \, | \, a_{[-k, k]})$ have the same marginals on the coordinates of the $r$-blocks $B$ contained in $[0, \ell[$. Moreover, if $a$ is chosen so that $\{\eta_{[-k, k]} = a_{[-k, k]}\}$ is a subset of $A$, we know that the positions of the $r$-blocks cover at least $(1-2\eps)\ell$ elements in $[0, \ell[$. Then, by considering the relative product of $\lambda_\ell(\cdot \, | \, a_{[-k, k]}, b_{[-m, 0]})$ and $\lambda_\ell(\cdot \, | \, a_{[-k, k]})$ over $\{\xi_{B_i}\}_{1 \leq i \leq p}$, we get:
$$\bar{d_\ell}\left(\lambda_\ell(\cdot \, | \, a_{[-k, k]}, b_{[-m, 0]}), \lambda_\ell(\cdot \, | \, a_{[-k, k]})\right) \leq 2 \varepsilon.$$
Finally, since $\mu(A) \geq 1-\varepsilon$, this yields 
$$\int \bar{d_\ell}\left(\lambda_\ell(\cdot \, | \, a_{[-k, k]}, b_{[-m, 0]}), \lambda_\ell(\cdot \, | \, a_{[-k, k]})\right) d\nu(a, b) \leq 3 \varepsilon.$$
\end{proof}

\begin{rmq}
We see that the proofs of Theorem \ref{thm:ornstein_K_weak_pinsker} and Theorem \ref{thm:bernoulli_weak_pinsker} are very similar. In both cases, we have a process $\xi$, whose conditional law given $\tau^n\xi$ is made of random blocks separated by deterministic blocks, and the random blocks are filled independently from each other. The main difference that prevents Ornstein's K-process from being Bernoulli is that the position of $r$-blocks is determined by the long sequences of <<$s$>>, and this creates correlations over long distances. But once we condition by $\tau^n \xi$, those sequences of <<$s$>> are entirely determined. Therefore we are left with filling independently all the $r$-blocks, and the past has no longer a significant influence on the future.

In that sense, when we look at the relative structure of Ornstein's K-process over $\tau^n$, the non-Bernoulli aspects disappear. However, when we look at the asymptotic properties of the weak Pinsker filtration obtained by applying $\{\tau^n\}_{n \geq 1}$, whether we start with a Bernoulli process or with a non-Bernoulli K-process, we get different results. Therefore, getting a better understanding of the classification of the various properties of weak Pinsker filtrations could help to develop a new classification of non-Bernoulli K-systems. 
\end{rmq}

\subsection*{Acknowledgements}
The author thanks Jean-Paul Thouvenot for the fruitful discussions, the insights and the ongoing exchanges regarding the present work.

\bibliographystyle{plain}
\bibliography{biblio_manuscrit}

\end{document}